\documentclass{amsart}
\usepackage{amsmath}
\usepackage{graphicx}
\usepackage{amsfonts}
\usepackage{amssymb}
\usepackage{pdfsync}
\usepackage{color}
\usepackage{mathrsfs}
\usepackage{epsfig}
\usepackage{url}
\usepackage{mathrsfs,a4wide}
\usepackage{url}
\usepackage{calligra}

\makeatletter
\def\@splitop#1#2\@nil{$\mathscr{#1}\!\!$\calligra#2\,\,}
\newcommand*\DeclareCursiveOperator[2]{%

  \newcommand#1{\mathop{\mbox{\@splitop#2\@nil}}\nolimits}}
\makeatother
\DeclareCursiveOperator{\Bnew}{B}
\DeclareCursiveOperator{\Cnew}{C}
\DeclareCursiveOperator{\Dnew}{D}
\DeclareCursiveOperator{\defe}{Def}

\theoremstyle{plain}
\newtheorem{theorem}{Theorem}[section]
\newtheorem{lemma}[theorem]{Lemma}
\newtheorem{proposition}[theorem]{Proposition}

\newtheorem{remark}[theorem]{Remark}

\theoremstyle{definition}
\numberwithin{equation}{section}

\newcommand{\inde}{\mathcal{I}}

\newcommand{\cs}{{\mathcal C}}

\newcommand{\is}{{\mathscr{I}}}

\newcommand{\rs}{{\mathcal R}}
\newcommand{\sss}{{\mathcal S}}
\newcommand{\qs}{\mathcal{Q}}

\newcommand{\E}{{\mathscr E}}

\newcommand{\F}{{\mathscr F}}

\newcommand{\W}{{\mathscr W}}

\newcommand{\R}{{\mathbb R}}
\newcommand{\N}{{\mathbb N}}
\newcommand{\Z}{{\mathbb Z}}

\newcommand{\ord}{\mathrm{ord}}

\newcommand{\ud}{\,\mathrm{d} }

\newcommand{\weakstar}{\stackrel{\star}{\rightharpoonup}}

\newcommand{\nper}[1]{|#1|_{\rm per}}

\newcommand{\e}{\varepsilon}
\newcommand{\ep}{\varepsilon}

\newcommand{\dist}{\mathrm{dist}}

\newcommand{\newatop}{\genfrac{}{}{0pt}{1}}

\newcommand{\X}{\mathcal{X}}

\newcommand{\mm}{\widehat m}

\newcommand{\AS}{\mathcal{AS}}

\newcommand{\bari}{\mathrm{bar}}
\newcommand{\XX}{\mathrm{X}}

\title
[Dynamics of screened particles towards equi-spaced ground states]{Dynamics of screened particles towards equi-spaced ground states}

\author[L. De Luca]
{Lucia De Luca}
\address[L. De Luca]{Istituto per le Applicazioni del Calcolo ``M. Picone'', IAC-CNR, Roma (Italy)}
\email[L. De Luca]{lucia.deluca@cnr.it}

\author[M. Goldman]
{Michael Goldman}
\address[M. Goldman]{CMAP, CNRS, \'Ecole Polytechnique, Institut Polytechnique de Paris, 91120 Palaiseau (France)} 
\email[M. Goldman]{michael.goldman@cnrs.fr}

\author[M. Ponsiglione]
{Marcello Ponsiglione}
\address[M. Ponsiglione]{Dipartimento di Matematica ``G. Castelnuovo'', Sapienza Universit\`a di Roma, Roma (Italy)} 
\email[M. Ponsiglione]{ponsigli@mat.uniroma1.it}

\usepackage{hyperref}
\begin{document}
\vskip .2truecm

\begin{abstract}
\small{
This paper deals with the  dynamics - driven by the gradient flow of negative fractional  seminorms - of empirical measures towards equi-spaced ground states. 

Specifically, we consider periodic empirical measures $\mu$ on the real line that are screened by the Lebesgue measure, i.e., with  $\mu-\!\ud x$ having zero average.
To each of these measures $\mu$ we associate a {(periodic)} function $u$ satisfying $u'=\!\ud x - \mu$. For $s\in (0,\frac 12)$ we introduce   energy functionals $\E^s(\mu)$ that can be understood as  the density of the  $s$-Gagliardo seminorm of $u$ per unit length.
Since for  $s\ge \frac 12$, the $s$-Gagliardo seminorms are infinite on functions with jumps, some regularization procedure is needed: For $s\in[\frac 12,1)$ we define $\E_\e^s(\mu):= \E^s(\mu_\e)$, where  $\mu_\e$ is obtained by mollifying $\mu$ on scale $\e$.

We prove that the minimizers of  $\E^s$ and $\E_\e^s$ are the equi-spaced configurations of particles with lattice spacing equal to one. Then, we prove the exponential convergence of the corresponding gradient flows to the  equi-spaced steady states. Finally, although for $s\in[\frac 12 ,1)$ the energy functionals $\E_\ep^s$ blow up as $\e\to 0$, their gradients are uniformly bounded (with respect to $\ep$), so that
the corresponding trajectories converge, as $\ep\to 0$, to  the gradient flow solution of a suitable renormalized energy. 
 

\vskip .3truecm \noindent \textsc{Keywords}: Fractional Seminorms; Periodic Minimizers; Gradient flows
\vskip.1truecm \noindent \textsc{Mathematics Subject Classification}: 
74N15, 
49J45,  
35R11 

}
\end{abstract}
\maketitle

{\small \tableofcontents}

\section*{Introduction}

{Periodic configurations are ubiquitous in nature, crystallization being a paradigmatic example in the physics of solids \cite{BL}. Here we are interested in studying the dynamics - driven by negative fractional semi-norms - towards the equi-spaced ground states for an infinite number of repelling  particles on the whole real line, mimicking charged particles screened  by a uniform field  of opposite charges.}

In our analysis, we bypass any compactness issue or boundary effect by working in a periodic setting;
specifically, given $\Lambda\in\N$, we  consider 
 $\Lambda$-periodic empirical measures $\mu$ on the real line that are screened by the Lebesgue measure, so that $\mu([0,\Lambda))=\Lambda$. 
To each $\mu$ we may associate the potential $u$ defined (up to additive constants) by $u'=\ud x - \mu$; notice that, since the particles are screened, $u$ is $\Lambda$-periodic.  

Given $s\in(0,\frac 12)$ we focus on the energy functionals
\begin{equation}\label{enesubintro}
 \E^{s}(\mu)
:=\frac 1 2\int_{0}^{\Lambda}\ud x \int_{\R}\frac{|u(x)-u(y)|^2}{|x-y|^{1+2s}}\ud y.
\end{equation}
Notice that the first integral is computed on the periodicity interval $(0,\Lambda)$, so that the functional  can be understood as  ($\Lambda$ times) the density of the squared $s$-Gagliardo seminorm of $u$ per unit length. 
We stress that for $s \in [ \frac 1 2,1)$ the energy in \eqref{enesubintro} is infinite for any empirical measure $\mu$. In order to overcome this issue, we adopt a regularization procedure by considering, for every $\ep>0$, the energy functional
\begin{equation}\label{enesubintro2}
\E^s_\ep(\mu):=\E^s(\mu\ast\rho_\ep);
\end{equation}
here $\rho_\ep(\cdot):=\frac{1}{\ep}\rho(\frac \cdot \ep)$ where $\rho$ is a standard  positive mollifier with support in $(-1,1)$.

The first result of the paper is that the global minimizers of the functionals $\E^s(\mu)$ and $\E^s_\ep(\mu)$ are given by the equi-spaced configurations with lattice spacing equal to one. This is a consequence of the fact that the energy functionals are convex with respect to the mutual distances between the particles. More precisely, for $s\in (0,\frac 12)$ the energy functionals are smooth in the space of all configurations without collisions, i.e.,  $\mu$ contains only Diracs of multiplicity one. For such configurations, which we call regular, the second variation of the energy functionals, with respect to the particle positions,  is strictly positive on all directions which are orthogonal to global translations. Notice that for this sub-critical range of the parameter $s$, the energy is  finite (and continuous) on all configurations of particles, but we prove that splitting any cluster of multiple particles decreases the energy, so that the global minimizer is a regular configuration. By convexity such a global minimizer is the unique - up to translations -  critical point of the energy. By symmetry arguments the equi-spaced configuration is a critical point, so that it is the only critical point and it coincides with the ground state of the energy. The situation for the super-critical cases $s\in [\frac 12, 1)$ is very similar, up to the fact that the regularized energy functionals are regular everywhere,
 and that we are able to prove the positiveness of the second variation only for configurations in which the mutual distances between the particles is at least of order $\e$. As a consequence, we still have that the equi-spaced lattice is the only ground state, but we cannot exclude the  presence of high energy stable configurations of closed packets of particles. 

Then, we focus on the dynamics driven by the gradient flow of the energy towards the equi-spaced ground states. Such an analysis
consists in proving that the trajectories of the particles  avoid collisions.  Let us discuss first the sub-critical case $s\in (0,\frac 12)$.
At a first glance, one could think that collisions are local maximum points of the energy, so that they are excluded by the simple fact that the energy decreases in time. Reality is a little more subtle: Since moving a cluster of particles could decrease the energy, there are trajectories where some of the particles collapse and translate remaining stack on each other, and such that the total energy decreases in time;  in order to exclude that these trajectories are solutions of the gradient flow, we have to show that, before the collision, splitting the particles provides a direction of steeper descent for the energy, so that the steepest descent directions of the energy landscape never lead to collisions. The situation is somehow easier in the supercritical case, where collisions cost an amount of energy that blows up as $\e\to 0$.
More precisely, 
Proposition \ref{quasigamma} provides the leading term $\sigma^s(\ep)$ (defined in \eqref{costosigma}) of the blowing up energy $\E_\ep^s$ induced by a single particle as $\e\to 0$.
On the other hand, the energy cost of a singularity with multiplicity $m$ behaves as $m^2\sigma^s(\ep)$, 
so that configurations with multiple singularities are clearly less favorable than regular ones.
 As a consequence, we have that, for any regular initial condition, for $\e$ small enough the trajectories avoid collisions and then converge to a ground state. Again, for initial data with closed packets of particles, we cannot exclude the presence of high energy stationary configurations.    

Finally, we show that in all cases the convergence to the equilibrium is exponentially fast. To this end, we stress that the energy functionals are invariant with respect to global translations  so that the ``periodic barycenter'' (in the sense of definition \eqref{barycenter}) of the evolving configuration remains constant  along the motion. Therefore, the convergence rate is determined by the first eigenvalue of the Hessian of the energy restricted to the space orthogonal to global translations.  

In Propositions \ref{varsecpossub} (for $s\in (0,\frac 1 2]$) and  \ref{varscepossuper} (for $s\in [\frac 1 2, 1)$) we prove that such an eigenvalue is bounded from above by  $\Gamma(s)\Lambda^{-2s}$ for some positive constant $\Gamma(s)$, for the equi-spaced configurations, and, for generic configurations, from below by $\gamma(s)\Lambda^{-2s}$ for some constant $\gamma(s)\le \Gamma(s)$. As a consequence, the - optimal in $\Lambda$ - convergence rate is proportional to  $\Lambda^{-2s}$.




A natural question is what happens (for $s\ge \frac 12$) to the energy functionals $\E_\ep^s$ and to the corresponding gradient flows as $\e\to 0$. 
First,  the  energies  $\W_\ep^s:=\E_\ep^s-\Lambda\sigma^s(\ep)$ converge locally uniformly in the space of regular configurations to a renormalized energy  $\W_0^s$ (as $\ep\to 0$). Moreover, while the energies $\E_\ep^s$ (for $s\in[\frac 1 2,1)$) blow up as $\e\to 0$,  their gradients (namely the slopes) stay locally uniformly bounded with respect to $\e$. As a consequence,  the gradient flows of $\E^s_\e$ (or equivalently of $\W_\ep^s$) converge, as $\e\to 0$, to the solution of  the gradient flow of the {\it renormalized energy} $\W_0^s$.

The case $s=\frac 1 2$ is of particular interest in the literature, since it corresponds to a Dirichlet-like bulk energy induced by a $\Lambda$-periodic distribution of topological singularities lying on a straight line  in the two dimensional plane. 
In this respect,  the energy functional $\W_0^s$ corresponds to the renormalized energy computed in the context of the Ginzburg-Landau vortices;
as shown in \cite{BS}, such a renormalized energy is given, up to additive and multiplicative constants depending only on $\Lambda$, by
$
-\sum_{\lambda\neq\lambda'}\log| \sin(\pi(x_\lambda-x_{\lambda'}))|,
$
and, in view of the convexity of the interaction potentials, its minimizers are the equi-spaced configurations. The renormalized energy has been introduced in several contexts also for infinite non-periodic configurations and it has been shown that the equi-spaced configuration, as well as any compact perturbation of it, is a ground state \cite{BS, L, PS, SS}. In particular, in \cite{SS} the renormalized energy represents the energy per unit length induced by a system of screened  charged particles lying on a straight line in the plane, once the infinite self energy of each particle is removed.
  
For the critical case $s=\frac 12$ the  energy functional in \eqref{enesubintro2} {can be seen as} a ``positive $\e$'' version of the renormalized energy considered in \cite{SS} for a   periodic distribution of screened charges.  Precisely, the potential $u$ {is} the trace of the harmonic conjugate of the potential generated by  $\ud x- \mu$, so that the functional $\E^{\frac{1}{2}}(\mu)$  represents the infinite energy density induced by the electric field generated by the screened particles, while $\E_\e^{\frac{1}{2}}(\mu)$ is its finite counterpart when the particles are diffused on the scale $\e$. More generally, for $s\in(0,1)$, our functionals are closely related to (one-dimensional) Riesz gases which have received considerable attention in the past few years, see \cite{Lewin,Serfatybook}.
In such a generality, we  show that the minimizers are still the equi-spaced configurations  also for $\ep>0$ (see \cite{PS} for the case $\ep=0$) and  that such ground states are attractors for the dynamics.

The emergence of periodic structures as a result of minimization of convex functionals has been much investigated in the last decades. In \cite{M93}
the minimization of the square of the $L^2$ norm has been considered, among functions having two opposite slopes. Such a result has been generalized in \cite{RW} to the case of two, possibly different, slopes. 
The case of fractional $\frac 12$-Gagliardo seminorm has been considered in \cite{GM},  again for functions with equal opposite slope;  their approach relies on a technique referred to as {\it reflection positivity} for which such  a symmetry assumption is somehow required. 
{In the aforementioned results, the functionals under minimization contains also a term penalizing the jumps of the slopes, which is multiplied by a certain (small) parameter. It is such a parameter, say $\delta$, that dictates the periodicity scale, that is proportional to $\delta^{\frac 1 3}$ in the ``local'' case of the $L^2$ norm and to $\delta^{\frac 1 2}$ for the $\frac 1 2$-Gagliardo seminorm.
 We stress that, for the ``local'' $L^2$-energy, the periodicity of minimizers (with the same scale) is proven also when the sharp penalization on the slope-jump is replaced by a Modica-Mortola functional.}

{In this paper we have adopted a more rigid approach: the slopes of the order parameter $u$ are either $1$ or $-\infty$ and, instead of a term penalizing the jumps of the slope, we have assumed that the region where the slope is $-\infty$ is quantized (the Dirac delta's have positive integer weights). 
This is strongly related to the framework considered in \cite{DPS},
} 
where a 
model for misfit dislocations at semi-coherent interfaces has been introduced and the optimality of equi-spaced dislocations has been proved. Since that model was one of the motivations for the present analysis, it is convenient to briefly describe the formalism   in \cite{DPS}.  In turn, the model studied there is motivated by   the modeling of misfit dislocations  by van der  Merwe in \cite{VdM}.  He considered semi-coherent straight interfaces between two  parallel square lattices with different spacing.  In his analysis it is tacitly assumed, as a well understood and unanimously accepted truth (as in the celebrated Read-Shockley paper \cite{RS} for small angle grain boundaries),  that dislocations are periodically distributed along the interface;  the optimal profile of the displacement corresponding to the periodic distribution of dislocations is provided and the corresponding interfacial energy is computed. Since the precise shape of the optimal profile only gives  lower order corrections in   the asymptotics of the energy density as the regularizing parameter $\e\to 0$, in \cite{DPS} a simplified model is introduced, where the transition is prescribed in a simple, non optimal way; namely, considering functions with two given slopes: A ``small'' positive slope of order $1$, accommodating elastically the lattice misfits, and a ``big'' negative slope of order $-1/\e$, providing the transition at the core length scale of the dislocation. Then, the elastic energy density induced by the resulting trace is given, up to pre-factors, by the (square of the) $\frac 12$-fractional seminorm, and the minimality of the equi-spaced configuration of dislocations is proved. A related model for semi-coherent interfaces has previously  been introduced and analyzed by $\Gamma$-convergence in \cite{FPS}; there, only the asymptotic (in the semi-coherent limit) uniform distribution of dislocations is proved, together with  the periodicity for minimizers of a suitable renormalized energy, corresponding to $\e=0$.

The main novelty of our analysis with respect to the results in \cite{DPS} is that  we also consider here the dynamics of misfit dislocations driven by the gradient flow of the induced elastic energy. To this end, it is convenient to replace the large negative slopes
with concentrated negative slopes represented by the empirical measure $\mu$.
On the one hand, this formalism yields after mollification
 merely a different (still non-optimal) profile; on the other hand, it fits naturally within a canonical framework based on the formalism of empirical measures, which is well-suited for studying the dynamics.


\vskip5pt
{\textsc{Acknowledgements:} LDL is a member of the Gruppo Nazionale per l'Analisi Matematica, la Probabilit\`a e le loro Applicazioni (GNAMPA) of the Istituto Nazionale di Alta Matematica (INdAM).
 
 LDL acknowledges the financial support of PRIN 2022HKBF5C ``Variational Analysis of complex systems in Materials Science, Physics and Biology'', PNRR Italia Domani, funded by the European Union via the program NextGenerationEU, CUP B53D23009290006.
 
MP acknowledges the financial support of PRIN 2022J4FYNJ  ``Variational methods for
stationary and evolution problems with singularities and interfaces", PNRR Italia Domani,
funded by the European Union via the program NextGenerationEU, CUP B53D23009320006.

Views and opinions expressed are however those of the authors only and do not necessarily reflect those of the European Union or The European Research Executive Agency. Neither the European Union nor the granting authority can be held responsible for them.

MG thanks the Sapienza University for its hospitality and support through a Visiting Professor position. Part of the research of MG was supported by the ANR STOIQUES }

%

\section{Setting of the problem}\label{Suno}
In this section we introduce the fractional seminorm  energies and the class of admissible functions we will deal with. 

\subsection{Function spaces}  
Let $\Lambda \in\N$ be fixed.
We set
$$
\X^{\Lambda}:= [0,\Lambda)^\Lambda.
$$
Notice that the entries  of a generic configuration $X\in\X^{\Lambda}$ are not required to be distinct.
For every $X=(x_1,\ldots,x_\Lambda)\in \X^{\Lambda}$ we will denote by $X^\ord =(x^\ord_1,\ldots,x^\ord_\Lambda) $ the {\it non-decreasing reordering} of $X$, namely the vector in $\X^\Lambda$ obtained permuting the entries of $X$ is such a way that $x^\ord_\lambda\le x^\ord_{\lambda+1}$ for every $\lambda=1,\ldots,\Lambda-1$.
 
Let $X\in\X^\Lambda$. We define $\sss(X):=\{\xi \in[0,\Lambda)\,:\,x_\lambda = \xi \textrm{ for some }\lambda=1,\ldots,\Lambda\}$.
To every $\xi\in\sss(X)$ we associate the set $\inde(\xi)$ of the indices $\lambda\in\{1,\ldots,\Lambda\}$ such that $x_{\lambda}=\xi$ and we define the 
 multiplicity $\mm(\xi)$ of $\xi$ as $\mm(\xi):=\sharp\inde(\xi)$.
 Notice that if $\mm(\xi)=1$, then $\inde(\xi)$ is made of one element, that will be denoted by $\lambda(\xi)$.
Moreover, if $X=X^{\ord}$, to any $\xi\in\sss(X)$ we can
 associate
the first index $\phi(\xi)$ for which $\xi$ is an entry of $X$, i.e.,
$\phi(\xi):=\min\{\lambda=1,\ldots,\Lambda\,:\,x_{\lambda}=\xi\}$;
notice that, in such a case,
$\mm(\xi)=\max\{m=1,\ldots,\Lambda\,:\, x_{\phi(\xi)+m-1}=\xi\}$.

We define the set of regular configurations $\rs^\Lambda$ as 
$$\rs^\Lambda:=\{X\in\X^\Lambda\,:\,\mm(\xi)=1\textrm{ for every }\xi\in\sss(X)\}.
$$
 Notice that $X\in\rs^\Lambda$ if and only if $\sharp \sss(X) = \Lambda$.

\vskip3pt
For every $X=(x_1,\ldots, x_\Lambda)\in\X^\Lambda$  we denote by $\XX$ its {\it $\Lambda$-periodic extension}, namely the $\Lambda$-periodic sequence $\XX:= \{\mathsf x_z\}_{z\in\Z}$  with $\mathsf x_z=  x_z$ for  $z=1, \ldots, \Lambda $.

We denote by $\{e_{\lambda}\}_{\lambda=1,\ldots,\Lambda}$ the canonical basis of $\R^\Lambda$ and we set $e:=\Lambda^{-\frac 1 2}\sum_{\lambda=1}^{\Lambda}e_\lambda$.

Moreover, for all $\tau\in \R$ we denote by $X+\tau$  the element of $\X^\Lambda$ whose $\Lambda$-periodic extension is $\XX + \tau$. 

We denote by  $\mu^{X}:=\sum_{z \in \Z} \delta_{\mathsf x_z}$ the empirical measure associated to the configuration $X$;
notice that $\mu^X\equiv\sum_{\xi\in\sss(X)}\mu^\xi$, where,
for every $\xi\in\sss(X)$, we have set
\begin{equation}\label{misurasing}
\mu^\xi:=\widehat{m}(\xi)\sum_{z\in\Z}\delta_{\xi+\Lambda z}.
\end{equation}
Finally, for every $X\in\X^\Lambda$, we
 introduce  the family $\AS(X)$ of functions {\it compatible} with $X$ as 
$$\AS(X):=\{u\in  BV_{\rm loc}(\R) \, :\,u'=\ud x- \mu^X\},$$
where $u'$  denotes the distributional derivative of $u$ and $\ud x$ is the standard Lebesgue measure.
By construction, the set  $\AS(X)$ contains functions differing just by additive constants. Since the energy functionals we  consider are insensitive to (horizontal and) vertical translations, with a little abuse of notation, we will denote by $u^X$ the ``unique up to additive constants'' element of $\AS(X)$.
Trivially, $u^X\equiv u^{X^\ord}$.

%
%
\subsection{The energy functionals}
 We introduce here the energy functionals that  will be  used throughout the paper.

{\bf The subcritical case $0<s<\frac 1 2$.}
For every $X\in\X^{\Lambda}$, we set
\begin{equation}\label{enesubcrit}
\E^{s}(X)
:=\frac 1 2\int_{0}^{\Lambda}\ud x\int_{\R}\frac{|u^X(x)-u^X(y)|^2}{|x-y|^{1+2s}}\ud y.
\end{equation}
Notice that the energy $\E^s$ is insensitive to global translations, i.e.,  for every $\tau\in\R$,
\begin{equation}\label{invariant}
\E^s (X+\tau )=\E^s(X);
\end{equation}
therefore, whenever it will be convenient, we will  assume without loss of generality  that $x_\lambda>0$ for every $\lambda=1,\ldots,\Lambda$.

For every $x\in\R$ we set $\bar u^X (x) := \frac{(u^X)^+(x) + (u^X)^-(x)}{2}$, where $(u^X)^\pm(x)$ denote the traces of the function $u^X$ at $x$, so that $\bar u^X$ coincides with $u^X$ on each of its continuity point. Then, we define
\begin{equation}\label{laplafra}
(-\Delta)^s u^X(x):=2\int_{\R}\frac{\bar u^X(x)-u^X(y)}{|x-y|^{1+2s}}\ud y,
\end{equation}
where the integral is intended in the sense of principal value if $x$ is a jump point for $u^X$.


\vskip5pt
{\bf The critical and supercritical case $\frac 1 2\le s<1$.}
Let $\rho$ be a standard mollifier supported in $(-1,1)$
 and, for any $\ep>0$, let $\rho_\ep(\cdot):=\frac{1}{\ep}\rho(\frac{\cdot}{\ep})$.
For every Radon measure $\nu$ on $\R$ we set $\nu_\e:= \rho_\e\ast\nu$. The notation $\nu_{\e\e}$ stands for $\rho_\e\ast \nu_\e = \rho_\e \ast (\rho_\e \ast \nu)$.
For every  $X\in\X^{\Lambda}$  we define
\begin{equation}\label{enesupercrit}
\E_\ep^{s}(X):=\frac 1 2\int_{0}^{\Lambda}\ud x\int_{\R}\frac{|u_\ep^X(x)-u_\ep^X(y)|^2}{|x-y|^{1+2s}}\ud y.
\end{equation}
We notice that for all $X\in \X^\Lambda$ the function $u_\ep^X$ is smooth and bounded, so that its $s$-fractional Laplacian
$\displaystyle (-\Delta)^s u_\ep^X(x):=2\int_{\R}\frac{u_\ep^X(x)-u_\ep^X(y)}{|x-y|^{1+2s}}\ud y$ is well defined in the sense of principal value.

\begin{remark}\label{bendefi}
\rm{
Notice that for every $X\in\X^\Lambda$ the operator in \eqref{laplafra} is well defined, in the sense of principal value, also for $\frac 1 2\le s<1$. Moreover, we have that 
\begin{equation}\label{convelapla}
\lim_{\ep\to 0}(-\Delta)^su^X_\ep(x)=(-\Delta)^su^X(x)
\qquad\textrm{for every }x\in\R.
\end{equation}
Finally, for any open set $U\subset\subset \rs^{\Lambda}$ (so that the minimal distance between two particles is uniformly bounded from below) we have that the convergence in \eqref{convelapla} at any $x\in\sss(X)$ is uniform in $U$.
}
\end{remark}
\section{First and second variations}
In this section we compute first and second variations of the energy functionals defined in \eqref{enesubcrit} and \eqref{enesupercrit}. 
For every $r>0$ and $x\in\R$ we denote by $B_r(x)=(x-r,x+r)$ the open interval centered at $x$ and having radius $r$; moreover, we set $B_r:=B_r(0)$.
Furthermore, for every $0<r<R$ and $x\in\R$ we set $A_{r,R}(x):=B_R(x)\setminus \overline{B}_r(x)$ and $A_{r,R}:=A_{r,R}(0)$. 
\subsection{The subcritical case $0<s<\frac 1 2$}
We start by computing the first and second variations of the energy $\E^s$ on regular configurations.
\begin{proposition}\label{prop:firstvarsub}
Let $0<s<\frac 1 2$ and let $X\in\rs^{\Lambda}$.
Let $\xi,\eta\in\sss(X)$ with $\xi,\eta>0$ and $\xi\neq \eta$.
Then
\begin{eqnarray}\label{prima}
{\partial_{x_{\lambda(\xi)}}} \E^{s}(X) &=& (-\Delta)^s u^X(\xi),\\ \label{seconda}
{\partial^2_{x^2_{\lambda(\xi)}}}\E^{s}(X) &=& 
\sum_{z\in\Z}\sum_{\xi'\in\sss(X)\setminus\{\xi\}}\frac{2}{|\Lambda z+\xi-\xi'|^{1+2s}},\\
\label{terza}
{\partial^2_{x_{\lambda(\xi)}\,x_{\lambda(\eta)}}}\E^{s}(X) &=& -
\sum_{z\in\Z}\frac{2}{|\Lambda z+\xi-\eta|^{1+2s}}.
\end{eqnarray}
\end{proposition}
\begin{proof}
We first prove \eqref{prima}. Let $h\neq 0$ and let  $I_{h}$ be the (for instance open) interval with extreme points $\xi$ and $\xi+h$. We set $A_{h}:=I_{h}+\Lambda\Z$.
By construction, for $|h|$ small enough, $X+h e_{\lambda(\xi)}\in\X^\Lambda$
and
 $u^{X+h e_{\lambda(\xi)}} =u^X + \text{sgn}(h) \chi_{A_h}$.
To simplify notation we consider only the case $h>0$. By direct computations
\begin{equation}\label{fvar1}
\begin{aligned}
&\E^s(X+he_{\lambda(\xi)})-\E^s(X)\\
=&\,\frac 1 2\int_{0}^\Lambda\ud x\int_{\R}\frac{(\chi_{A_{h}}(x)-\chi_{A_h}(y))\big((2u^X(x)+\chi_{A_{h}}(x))-(2u^{X}(y)+\chi_{A_{h}}(y))\big)}{|x-y|^{1+2s}}\ud y\\
=&\,\int_{0}^\Lambda\ud x\int_{\R}\frac{\chi_{A_{h}}(x)\big((2u^X(x)+\chi_{A_{h}}(x))-(2u^{X}(y)+\chi_{A_{h}}(y))\big)}{|x-y|^{1+2s}}\ud y\\
=&\,2\int_{I_h} \ud x
\int_{\R}\frac{( u^X(x)+\frac 12)-(u^{X}(y)+ \frac 12 \chi_{A_{h}}(y))}{|x-y|^{1+2s}}\ud y\\
=&\,2\int_{I_h} \ud x
\int_{\R\setminus \overline I_h}\frac{( u^X(x)+\frac 12)-(u^{X}(y)+ \frac 12 \chi_{A_{h}}(y))}{|x-y|^{1+2s}}\ud y\, .
\end{aligned}
\end{equation}
The first equality is nothing but  computing the difference of squares. The second  equality is a consequence of  periodicity, using the change of variable $x\mapsto x-\Lambda z$, $y\mapsto y-\Lambda z$, together with Fubini Theorem as follows
\begin{equation*}
\begin{aligned}
&\,\frac 1 2\int_{0}^\Lambda\ud x\int_{\R}\frac{-\chi_{A_h}(y)\big((2u^X(x)+\chi_{A_{h}}(x))-(2u^{X}(y)+\chi_{A_{h}}(y))\big)}{|x-y|^{1+2s}}\ud y\\
=&\,\frac 1 2\int_{0}^\Lambda\ud x\int_{\R}\frac{\chi_{A_h}(y)\big((2u^X(y)+\chi_{A_{h}}(y))-(2u^{X}(x)+\chi_{A_{h}}(x))\big)}{|x-y|^{1+2s}}\ud y\\
=&\,\sum_{z\in\Z}\frac 1 2\int_{0}^\Lambda\ud x\int_{\Lambda z}^{\Lambda(z+1)}\frac{\chi_{A_h}(y)\big((2u^X(y)+\chi_{A_{h}}(y))-(2u^{X}(x)+\chi_{A_{h}}(x))\big)}{|x-y|^{1+2s}}\ud y\\
=&\,\sum_{z\in\Z}\frac 1 2\int_{-\Lambda z}^{\Lambda(1-z)}\ud x\int_{0}^{\Lambda}\frac{\chi_{A_h}(y)\big((2u^X(y)+\chi_{A_{h}}(y))-(2u^{X}(x)+\chi_{A_{h}}(x))\big)}{|x-y|^{1+2s}}\ud y\\
=&\,\frac 1 2\int_{0}^{\Lambda} \ud y\int_{\R}\frac{\chi_{A_h}(y)\big((2u^X(y)+\chi_{A_{h}}(y))-(2u^{X}(x)+\chi_{A_{h}}(x))\big)}{|x-y|^{1+2s}}\ud x\, .
\end{aligned}
\end{equation*}
Finally, the third equality  in \eqref{fvar1} is immediate, while the last one follows from the fact that
\begin{equation}\label{vienezero}
\begin{aligned}
&\,\int_{I_h} \ud x
\int_{I_h}\frac{( u^X(x)+\frac 12)-(u^{X}(y)+ \frac 12 \chi_{A_{h}}(y))}{|x-y|^{1+2s}}\ud y\\
=&\,\int_{I_h} \ud x
\int_{I_h}\frac{u^X(x)-u^{X}(y)}{|x-y|^{1+2s}}\ud y
=\int_{I_h} \ud x
\int_{I_h}\frac{x-y}{|x-y|^{1+2s}}\ud y=0.
\end{aligned}
\end{equation}
We set $\xi^h:=\xi+\frac h 2$, so that $I_h=B_{\frac{h}{2}}(\xi^h)$. Let $0<\sigma<\min_{\xi'\in\sss(X)\setminus\{\xi\}}|\xi-\xi'|$ and notice that, for $h$ small enough
\begin{eqnarray*}
u^X(x)-u^{X}(y)=&\,x-y&\qquad\textrm{if }x\in I_h,\,\,y\in\Big(\xi_h+\frac h 2, \xi_h+\sigma\Big),\\
u^X(x)-u^{X}(y)=&\,x-y-1&\qquad\textrm{if }x\in I_h,\,\,y\in\Big(\xi_h-\sigma, \xi_h-\frac h 2\Big).
\end{eqnarray*}
Whence we deduce
\begin{equation}\label{fvar2}
\begin{aligned}
&\int_{I_h} \ud x\int_{B_{\sigma}(\xi^h)\setminus \overline I_h  }\frac{( u^X(x)+\frac 12)-(u^{X}(y)+ \frac 12 \chi_{A_{h}}(y))}{|x-y|^{1+2s}}\ud y\\
=&\,\int_{B_{\frac h 2}(\xi^h)}\ud x\int_{B_{\sigma}(\xi^h)\setminus \overline B_{\frac h 2}(\xi^h)}\frac{ u^X(x)+\frac 12-u^{X}(y)}{|x-y|^{1+2s}}\ud y=0.
\end{aligned}
\end{equation}
Therefore, in view of \eqref{fvar1} and \eqref{fvar2}, we deduce
\begin{equation}\label{fvar3}
\begin{aligned}
&\E^s(X+he_{\lambda(\xi)})-\E^s(X)\\
=&\,2\int_{B_{\frac h 2}(\xi^h)}\ud x\int_{\R\setminus \overline B_{\sigma}(\xi^h)}\frac{( u^X(x)+\frac 12)-(u^{X}(y)+ \frac 12 \chi_{A_{h}}(y))}{|x-y|^{1+2s}}\ud y\\
=&\,2\int_{B_{\frac h 2}(\xi^h)}\ud x\int_{\R\setminus \overline B_{\sigma}(\xi^h)}\frac{u^X(x)+\frac 12-u^{X}(y)}{|x-y|^{1+2s}}\ud y +\mathrm{r}(h)\,,
\end{aligned}
\end{equation}
where (for $h$ small enough)
\begin{equation}\label{perfvar3}
\begin{aligned}
r(h) : = &\,\int_{B_{\frac h 2}(\xi^h)}\ud x\int_{\R\setminus \overline B_{\sigma}(\xi^h)}\frac{\chi_{A_{h}}(y)}{|x-y|^{1+2s}}\ud y\\
&=\,
\int_{B_{\frac h 2}(\xi^h)}\ud x\sum_{z\in\Z\setminus\{0\}}\int_{B_{\frac h 2}(\xi^h)+\Lambda z}\frac{1}{|x-y|^{1+2s}}\ud y\le h^2\sum_{z\in\Z\setminus\{0\}}\Big(\frac{2}{\Lambda |z|}\Big)^{1+2s}.
\end{aligned}
\end{equation}
Let  us now consider the change of variable $x'= \frac{x-\xi^h}{h}$ so that $x=\xi^h+h x'$. By \eqref{fvar3} and \eqref{perfvar3}, using the Dominated Convergence Theorem, we get
\begin{equation}\label{fvar4}
\begin{aligned}
&\lim_{h\to 0}\frac{\E^s(X+he_{\lambda(\xi)})-\E^s(X)}{h}\\
=&\,\lim_{h\to 0}\Big[2\int_{-\frac{1}{2}}^{\frac 1 2}\ud x'\int_{\R\setminus \overline B_{\sigma}(\xi^h)}\frac{u^X(\xi^h+h x')+\frac 12-u^{X}(y)}{|\xi^h+h x' -y|^{1+2s}}\ud y +\frac{r(h)}{h}\Big]\\
=&\,2\int_{-\frac{1}{2}}^{\frac 1 2}\ud x'\int_{\R\setminus \overline B_{\sigma}(\xi)}\frac{u^X(\xi)+\frac 12-u^{X}(y)}{|\xi-y|^{1+2s}}\ud y ,
\end{aligned}
\end{equation}
which, by the arbitrariness of $\sigma$, and by the very definition in \eqref{laplafra}, implies \eqref{prima}.

Now we prove \eqref{seconda}. By \eqref{prima}, using the invariance by global translations, we have
\begin{equation}\label{seconda0}
\begin{aligned}
\partial^2_{x^2_{\lambda(\xi)}}\E^s(X)=&\,\lim_{h\to 0}\frac{(-\Delta)^su^{X+he_{\lambda(\xi)}}(\xi + h)-(-\Delta)^su^{X}(\xi)}{h}\\
=&\,\lim_{h\to 0}\frac{(-\Delta)^su^{X-h\hat e_{\lambda(\xi)}}(\xi)-(-\Delta)^su^{X}(\xi)}{h},
\end{aligned}
\end{equation}
 where $\hat e_{\lambda(\xi)}$ denotes the vector with the $\lambda(\xi)$-th entry equal to $0$ and all the remaining ones equal to $1$. 
To simplify notation we consider only the case $h>0$ and 
we set $\hat I_h:=\bigcup_{\xi'\in\sss(X)\setminus\{\xi\}}(\xi'-h,\xi')$ and
$\hat A_h:= \hat I_h+\Lambda\Z$.
Then, for $h$ small enough, $u^{X-h\hat e_{\lambda(\xi)}}\equiv u^X-\chi_{\hat A_h}$.
By the very definition of fractional Laplacian in \eqref{laplafra} and by the Fundamental Theorem of Calculus, we get
\begin{equation*}
\begin{aligned}
\lim_{h\to 0}\frac{(-\Delta)^su^{X-h\hat e_{\lambda(\xi)}}(\xi)-(-\Delta)^su^{X}(\xi)}{h}=&\,\lim_{h\to 0}\frac{2}{h}\int_{\hat A_h}\frac{1}{|\xi-y|^{1+2s}}\ud y\\
=&\,\sum_{z\in\Z}\sum_{\xi'\in\sss(X)\setminus\{\xi\}}\frac{2}{|\Lambda z+\xi-\xi'|^{1+2s}},
\end{aligned}
\end{equation*}
which, together with \eqref{seconda0}, yields \eqref{seconda}.\\
We finally turn to \eqref{terza}. 
By \eqref{prima}, we have
\begin{equation}\label{terza0}
 {\partial^2_{x_{\lambda(\xi)}\,x_{\lambda(\eta)}}}\E^{s}(X) = \lim_{h\to 0} \frac{(-\Delta)^s u^{X+h e_\lambda(\eta)}(\xi)- (-\Delta)^s u^{X}(\xi)}{h}.
\end{equation}
To simplify notation we focus on the case $h>0$; for $h$ small enough, 
$$
(-\Delta)^s u^{X+h e_\lambda(\eta)}=(-\Delta)^s u^{X}+(-\Delta)^s \chi_{E_h},
$$ 
where we have set $E_h:=\cup_{z\in \Z} (\eta+\Lambda z,\eta+\Lambda z+h)$.  Therefore, by \eqref{terza0}, we obtain
\begin{equation*}
 {\partial^2_{x_{\lambda(\xi)}\,x_{\lambda(\eta)}}}\E^{s}(X) 
 = \lim_{h\to 0} \frac{(-\Delta)^s \chi_{E_h}(\xi)}{h}
=-2 \lim_{h\to 0} \frac{1}{h}\int_{E_h} \frac{1}{|\xi-y|^{1+2s}}\ud y=-\sum_{z\in\Z}\frac{2}{|\Lambda z+\xi-\eta|^{1+2s}}.
\end{equation*}
This concludes the proof of \eqref{terza}.
\end{proof}
In the next proposition, we prove that the Hessian $\nabla^2\E^s$ is positive definite on $\rs^\Lambda$, providing a lower bound on its first eigenvalue.
 
We define the class $\cs^\Lambda$ of equispaced configurations as
\begin{equation}\label{crit}
\cs^\Lambda:=\{X\in\rs^\Lambda\,:\,\sss(X)=\{\tau,1+\tau,\ldots,\Lambda-1+\tau\}\,, \quad\tau\in[0,1)\}.
\end{equation}
\begin{proposition}\label{varsecpossub}
Let $0<s<\frac 1 2$ and $\Lambda\in\N$. Then, for every $X\in\rs^\Lambda$ with $\sss(X)\subset(0,\Lambda)$ and for every $h\in\R^\Lambda$
\begin{equation}\label{posvarsecsub}
\langle\nabla^2\E^s(X)h,h\rangle=\sum_{\newatop{\xi,\eta\in\sss(X)}{\xi\neq\eta}}a^s(\xi,\eta)(h_{\lambda(\xi)}-h_{\lambda(\eta)})^2,
\end{equation}
where
\begin{equation}\label{coeffa}
a^s(\xi,\eta):=2\sum_{z\in\Z}\frac{1}{|\xi-\eta+\Lambda z|^{1+2s}}\qquad\textrm{for every }\xi,\eta\in \sss(X) \textrm{ with }\xi\neq\eta.
\end{equation}
As a consequence, if $h\cdot e=\sum_{\lambda=1}^{\Lambda}h_\lambda=0$, then
\begin{equation}\label{stimasubpeggio}
\langle\nabla^2\E^s(X)h,h\rangle\ge \gamma(s)\Lambda^{-2s}|h|^2,
\end{equation}
for some positive constant $\gamma(s)$ depending only on $s$.

Moreover,
 for every $X\in\cs^{\Lambda}$, there exists $h\in\R^\Lambda$ with $h\neq0$, $h\cdot e=0$ and
\begin{equation}\label{stimasubmeglio}
\langle\nabla^2\E^s(X)h,h\rangle\le \Gamma(s)\Lambda^{-2s}|h|^2,
\end{equation}
for some positive constant $\Gamma(s)$ depending only on $s$.

\end{proposition}
\begin{proof}
Identity \eqref{posvarsecsub} is an immediate consequence of  Proposition \ref{prop:firstvarsub}.
Now we prove \eqref{stimasubpeggio}.
By definition, for every $\xi,\eta\in\sss(X)$ with $\xi\neq\eta$ we have
\begin{equation}\label{stimaovvia}
a^s(\xi,\eta)\ge \gamma(s)\Lambda^{-1-2s}.
\end{equation}
Now,  since
\[
 \sum_{\newatop{\lambda,\lambda'\in\{1,\ldots,\Lambda\}}{\lambda\neq\lambda'}} (h_{\lambda}-h_{\lambda'})^2= (\Lambda-1) \sum_{\lambda=1}^{\Lambda} h^2_{\lambda}- 2\sum_{\newatop{\lambda,\lambda'\in\{1,\ldots,\Lambda\}}{\lambda\neq\lambda'}}  h_{\lambda} h_{\lambda'}
\]
and
\[
 \Big(\sum_{\lambda=1}^{\Lambda} h_{\lambda}\Big)^2= \sum_{\lambda=1}^{\Lambda} h^2_{\lambda} +2\sum_{\newatop{\lambda,\lambda'\in\{1,\ldots,\Lambda\}}{\lambda\neq\lambda'}} h_{\lambda} h_{\lambda'},
\]
we find that for every $h\in\R^\Lambda$ with $h\cdot e=0$,
\begin{equation}\label{stimaovvia2}
 \sum_{\newatop{\lambda,\lambda'\in\{1,\ldots,\Lambda\}}{\lambda\neq\lambda'}} (h_{\lambda}-h_{\lambda'})^2= \Lambda \sum_{\lambda=1}^{\Lambda} h^2_{\lambda}.
\end{equation}
Thus, by \eqref{stimaovvia} and \eqref{stimaovvia2}, we have
\begin{equation}\label{generallambda1}
 \sum_{\newatop{\xi,\eta\in\sss(X)}{\xi\neq \eta}} a^s(\xi,\eta)(h_{\lambda(\xi)}-h_{\lambda(\eta)})^2\geq \gamma(s)\Lambda^{-1-2s}
 \sum_{\newatop{\lambda,\lambda'\in\{1,\ldots,\Lambda\}}{\lambda\neq\lambda'}} (h_{\lambda}-h_{\lambda'})^2= \gamma(s)
 \Lambda^{-2s}|h|^2,
\end{equation}
whence the claim follows by \eqref{posvarsecsub}.

Now we prove the upper bound.
To this end let 
\begin{equation}\label{nperdef}
\nper{\cdot}:=\min_{z\in\Z}|\cdot+\Lambda z|
\end{equation}
denote the ($\Lambda$-)periodic norm on $\R$.
We define the vector $h\in\R^\Lambda$ by setting
\[
 h_\lambda:= \nper{\lambda}-\frac{1}{\Lambda}\sum_{\lambda'=0}^{\Lambda-1} \nper{\lambda'}.
\]
By construction $h\cdot e=0$.
On the one hand we have
\begin{equation*}\label{lowerboundh2}
 |h|^2\ge C \Lambda^3.
\end{equation*}
On the other hand, since
\[
 a^s(\lambda,\lambda')\le \Gamma(s) \frac{1}{\nper{\lambda-\lambda'}^{1+2s}},
\]
we have
\[
\begin{aligned}
\langle\nabla^2\E^s(X)h,h\rangle&\le \Gamma(s)\sum_{\newatop{\lambda,\lambda'\in\{0,1,\ldots,\Lambda-1\}}{\lambda\neq\lambda'}}\frac{(\nper{\lambda}-\nper{\lambda'})^2}{\nper{\lambda-\lambda'}^{1+2s}}\\
&\le\Gamma(s)\sum_{\newatop{\lambda,\lambda'\in\{0,1,\ldots,\Lambda-1\}}{\lambda\neq\lambda'}}\nper{\lambda-\lambda'}^{1-2s}\\
&\le \Gamma(s) \Lambda^{3-2s}.
\end{aligned}
\]
Combining both estimates together yields \eqref{stimasubmeglio}.

\end{proof}
\begin{remark}\label{subcomesuper}
\rm{
Let us point out that in order to prove \eqref{stimasubmeglio}, being $0<s<1/2$, it would have been enough to consider the interpolating  function $\hat h$ to be piecewise constant instead of piecewise affine. 
The advantage of introducing piecewise affine functions is that 
they belong to $H^s$ also for $s\ge 1/2$,  covering the whole range of exponents $0<s<1$.
In fact, following along the lines of the proof of Proposition \ref{varsecpossub}, 
we have immediately that, also for $\frac1 2\le s <1$
\begin{equation}\label{darichiamareinsuper}
\sum_{\newatop{\xi,\eta\in\sss(X)}{\xi\neq\eta}}a^s(\xi,\eta)\big(h_{\lambda(\xi)}-h_{\lambda(\eta)}\big)^2\ge \gamma(s)\Lambda^{-2s}|h|^2 \textrm{ for every }X\in\rs^\Lambda,\,\, h\in\R^\Lambda\textrm{ with }h\cdot e=0
\end{equation}
and
\begin{equation}\label{darichiamareinsuper2}
\sum_{\newatop{\xi,\eta\in\sss(X)}{\xi\neq\eta}}a^s(\xi,\eta)\big(h_{\lambda(\xi)}-h_{\lambda(\eta)}\big)^2\le \Gamma(s)\Lambda^{-2s}|h|^2 \textrm{ for every }X\in\cs^\Lambda,\,\, \textrm{for some } h\in\R^\Lambda\textrm{ with }h\cdot e=0,
\end{equation}
where the coefficients $a^s(\xi,\eta)$ are defined in \eqref{coeffa}.}
\end{remark}

Arguing verbatim as in the proof of \eqref{prima} in Proposition \ref{prop:firstvarsub} we get its 
 extension  to the case of multiple singularities (we recall that $\mathcal{I}(\xi)$ is the set of indices of the particles at position $\xi$ and $\widehat m (\xi)$ is its cardinality).

\begin{lemma}\label{missinglemma}
Let $0<s<\frac 12$ and let $X\in\X^\Lambda$; then, given $\{h_\xi\}_{\xi \in \sss(X)}\subset \R^\Lambda$, we have
\[
 \E^s \Big(X+ \sum_{\xi\in \sss(X)} h_\xi\sum_{m\in \mathcal{I}(\xi)} e_m \Big)= \sum_{\xi\in \sss(X)} h_\xi \widehat{m}(\xi) (-\Delta)^s u^X(\xi) + \mathrm o\Big (\sum_{\xi\in \sss(X)} |h_\xi| \Big)\, .
\]
\end{lemma}
The next lemma shows that splitting multiple particles gives a direction with infinite slope for the energy.
\begin{lemma}\label{missinglemma2}
Let $0<s<\frac 12$ and let $X\in\X^\Lambda\setminus\rs^\Lambda$. Let $\xi\in\sss(X)$ be such that $\mm(\xi)\ge 2$. Then, for every $m\in\inde(\xi)$, we have
$$
\partial_{x_m}^{\pm}\E^s(X)=\mp \infty.
$$
\end{lemma}
\begin{proof}
Fix $m\in\inde(\xi)$. By construction, for $|h|$ small enough, $u^{X+h e_m} =u^X + \text{sgn}(h) \chi_{I_h}$, where $I_{h}$ is the (for instance open) interval with extreme points $\xi$ and $\xi+h$.
We set $A_{h}:=I_{h}+\Lambda\Z$\,.
To simplify notation we consider only the case $h>0$. By arguing verbatim as in \eqref{fvar1} we have
\begin{equation}\label{mancava1}
\E^s(X+he_{m})-\E^s(X)=2\int_{I_h}\ud x\int_{\R\setminus \overline I_h}\frac{u^X(x)+\frac 1 2- \big( u^X(y) + \frac12 \chi_{A_h}(y) \big)}{|x-y|^{1+2s}}\ud y.
\end{equation}
We set $\xi^h:=\xi+\frac h 2$, so that $I_h=B_{\frac{h}{2}}(\xi^h)$. Let $0<\sigma<\min_{\xi'\in\sss(X)\setminus\{\xi\}}|\xi-\xi'|$;  we notice that, for $h$ small enough
\begin{eqnarray*}
u^X(x)-u^{X}(y)=&\,x-y&\qquad\textrm{if }x\in I_h,\,y\in\Big( \xi_h+\frac h 2,\xi_h+\sigma\Big),\\
u^X(x)-u^{X}(y)=&\,x-y-\mm(\xi)  &\qquad\textrm{if }x\in I_h,\, y\in\Big(\xi_h-\sigma,\xi_h-\frac h 2\Big),
\end{eqnarray*}
whence we deduce
\begin{equation}\label{fvar2dopo}
\begin{aligned}
&\int_{B_{\frac h 2}(\xi^h)}\ud x\int_{B_{\sigma}(\xi^h)\setminus \overline B_{\frac h 2}(\xi^h)}\frac{( u^X(x)+\frac 12)-(u^{X}(y)+ \frac 12 \chi_{A_{h}}(y))}{|x-y|^{1+2s}}\ud y\\
=&\,\int_{B_{\frac h 2}(\xi^h)}\ud x\int_{B_{\sigma}(\xi^h)\setminus \overline B_{\frac h 2}(\xi^h)}\frac{ u^X(x)+\frac 12-u^{X}(y)}{|x-y|^{1+2s}}\ud y\\
=&\,(1-\mm(\xi))\int_{B_{\frac h 2}(\xi^h)}\ud x\int_{\xi^h-\sigma}^{\xi^h-\frac h 2}\frac{1}{|x-y|^{1+2s}}\ud y\\
=&\,(1-\mm(\xi))\frac{1}{2s(1-2s)}\Big(h^{1-2s}-\Big(\sigma + \frac h 2 \Big)^{1-2s}+\Big ( \sigma - \frac h2 \Big)^{1-2s}\Big)\\
=&\,
(1-\mm(\xi))\frac{h^{1-2s}}{2s(1-2s)}+
\mathrm{O}(h).
\end{aligned}
\end{equation}
Finally, since the numerator in the integral below is in $L^\infty(\R^2)$,
one can easily prove that
\begin{equation*}
\int_{B_{\frac h 2}(\xi^h)}\ud x\int_{\R\setminus\overline{B}_{\sigma}(\xi^h)}\frac{( u^X(x)+\frac 12)-(u^{X}(y)+ \frac 12 \chi_{A_{h}}(y))}{|x-y|^{1+2s}}\ud y=\mathrm{O}(h),
\end{equation*}
which, together with \eqref{mancava1} and \eqref{fvar2dopo}, yields
\begin{equation*}
\E^s(X+he_{m})-\E^s(X)=(1-\mm(\xi))\frac{h^{1-2s}}{s(1-2s)}+\mathrm{O}(h),
\end{equation*}
whence the claim follows, since $\mm(\xi)\ge 2$.
\end{proof}
Now, according with Lemma \ref{missinglemma2}, we prove that  separating at least one singularity from  a close packed cluster of  singularities   decreases the energy.  
For every pair of sets $A,B\subset\R$, we set
\begin{equation}\label{interazione}
\is^s(A,B):=\int_{A}\ud x\int_{B}\frac{1}{|x-y|^{1+2s}}\ud y.
\end{equation}
\begin{proposition}\label{spostouna}
Let $0<s<\frac 1 2$, $\sigma>0$  and let $X\in\rs^\Lambda$. Let $\qs:=\{\xi_1,\ldots,\xi_K\}\subset\sss(X)$ with $K\ge 2$ and $0<\xi_k\le \xi_{k+1}$ (for all $k=1,\ldots,K-1$) be such that  
$$
\min_{\newatop{\xi\in \sss(X)\setminus \qs}{k=1,\ldots,K}}\{ |\xi-\xi_k|, \, |\xi-\xi_k - \Lambda| , \, |\xi-\xi_k + \Lambda| \} \ge \sigma\, . 
$$
Then, setting $\delta:= \xi_K-\xi_1$, we have 
\begin{equation*}
\nabla \E^s(X)\cdot e_{\lambda(\xi_K)}
\le \frac{1-K}{2s}\delta^{-2s}+C (\Lambda,s,\sigma),
\end{equation*}
for some constant $C(\Lambda,s,\sigma)$ (independent of $\delta$).
\end{proposition}
\begin{proof}
Let $h>0$. We set $I_h:=(\xi_K,\xi_{K}+h)$ and $A_h:=I_h+\Lambda\Z$. By \eqref{fvar1}, for $h$ small enough we have
\begin{equation}\label{comefvar1}
\E^s(X+he_{\lambda(\xi_K)})-\E^s(X)=2\int_{I_h} \ud x
\int_{\R\setminus \overline I_h}\frac{( u^X(x)+\frac 12)-(u^{X}(y)+ \frac 12 \chi_{A_{h}}(y))}{|x-y|^{1+2s}}\ud y,
\end{equation}
and
\begin{equation}\label{finito}
\int_{I_h} \ud x
\int_{\R\setminus [\xi_1-\sigma,\xi_K+\sigma]}\frac{(u^X(x)+\frac 12)-(u^{X}(y)+ \frac 12 \chi_{A_{h}}(y))}{|x-y|^{1+2s}}\ud y\le C(\Lambda,\sigma)h\sum_{z\in\Z\setminus\{0\}} \frac{1}{|z|^{1+2s}}.
\end{equation}
Therefore, setting
\begin{equation*}
E^s_\delta[h]:=
\int_{I_h} \ud x
\int_{(\xi_1-\sigma,\xi_K+\sigma)\setminus\overline I_h}\frac{(u^X(x)+\frac 12)-(u^{X}(y)+ \frac 12 \chi_{A_{h}}(y))}{|x-y|^{1+2s}}\ud y,
\end{equation*}
in order to prove the claim, it is enough to show that
\begin{equation}\label{nuocla}
\limsup_{h\to 0^+}\frac{E^s_\delta[h]}{h}\le \frac{1-K}{2s}\delta^{-2s}+C(s,\sigma) \, .
\end{equation}
We first write (for $h$ small enough)  $E^s_\delta[h]:=E^s_{\delta,1}[h]+E^s_{\delta,2}[h]$, where 
\begin{equation*}
\begin{aligned}
E^s_{\delta,1}[h]=&\,\int_{I_h} \ud x
\int_{\xi_1}^{\xi_K}\frac{u^X(x)+\frac 12-u^{X}(y)}{|x-y|^{1+2s}}\ud y+\int_{I_h} \ud x
\int_{\xi_K+h}^{\xi_K+\delta}\frac{u^X(x)+\frac 12-u^{X}(y)}{|x-y|^{1+2s}}\ud y,\\
E^s_{\delta,2}[h]
:=&\,\int_{I_h} \ud x
\int_{\xi_1-\sigma}^{\xi_1}\frac{u^X(x)+\frac 12-u^{X}(y)}{|x-y|^{1+2s}}\ud y+\int_{I_h} \ud x
\int_{\xi_K+\delta}^{\xi_K+\sigma}\frac{u^X(x)+\frac 12-u^{X}(y)}{|x-y|^{1+2s}}\ud y.
\end{aligned}
\end{equation*}
Now we observe that
\begin{eqnarray}\label{usodopo0}
u^X(x)-u^{X}(y)=&\,-K+x-y\qquad&\textrm{whenever }x\in I_h, y\in (\xi_1-\sigma,\xi_{1}),\\
\label{solouna}
u^X(x)-u^{X}(y)\le&\,-1+x-y \qquad&\textrm{whenever }x\in I_h, y\in (\xi_1,\xi_{K})
\\ \label{usodopo}
u^X(x)-u^{X}(y)=&\,x-y\qquad&\textrm{whenever }x\in I_h, y\in(\xi_K+h,\xi_K+\sigma)\,.
\end{eqnarray}
Hence, by \eqref{solouna} and \eqref{usodopo}, recalling the definition of $\is^s$ in \eqref{interazione} and using Taylor expansion, we get 
\begin{equation}\label{termuno}
\begin{aligned}
E^s_{\delta,1}[h]
\le &\, -\frac{1}{2}\is^s(I_h,(\xi_1,\xi_K))+\frac{1}{2}\is^s(I_h,(\xi_K+h,\xi_K+\delta))\\
&\,+\int_{I_h}\ud x\int_{\xi_1}^{\xi_K}(x-y)^{-2s}\ud y-\int_{I_h}\ud x\int_{\xi_K+h}^{\xi_K+\delta}(y-x)^{-2s}\ud y\\
=&\,-\frac{1}{2}\frac{1}{2s(1-2s)}\big(h^{1-2s}+\delta^{1-2s}-(h+\delta)^{1-2s}\big)\\
&\,+\frac{1}{2}\frac{1}{2s(1-2s)}\big(h^{1-2s}+(\delta-h)^{1-2s}-\delta^{1-2s}\big)\\
&\,+\frac{1}{(1-2s)(2-2s)}\big(-h^{2-2s}-\delta^{2-2s}+(h+\delta)^{2-2s}\big)\\
&\,-\frac{1}{(1-2s)(2-2s)}\big(-h^{2-2s}-(\delta-h)^{2-2s}+\delta^{2-2s}\big)\\
=&\,\frac{1}{2}\frac{1}{2s(1-2s)}\big((\delta+h)^{1-2s}+(\delta-h)^{1-2s}-2\delta^{1-2s}\big)\\
&\,+\frac{1}{(1-2s)(2-2s)}\big((h+\delta)^{2-2s}+(\delta-h)^{2-2s}-2\delta^{2-2s}\big)\\
=&\, \mathrm{O}(h^2),
\end{aligned}
\end{equation}
where $|\mathrm{O}(h^2)|\le C(s, \delta) h^2$.
Analogously, by \eqref{usodopo0} and \eqref{usodopo}  we have 
\begin{equation}\label{termdue0}
\begin{aligned}
E^s_{\delta,2}[h]
=&\, \Big(\frac 1 2-K\Big)\is^s(I_h,(\xi_1-\sigma,\xi_1))+\frac{1}{2}\is^s(I_h,(\xi_K+\delta,\xi_K+\sigma))\\
&\,+\int_{I_h}\ud x\int_{\xi_1-\sigma}^{\xi_1}(x-y)^{-2s}\ud y-\int_{I_h}\ud x\int_{\xi_K+\delta}^{\xi_K+\sigma}(y-x)^{-2s}\ud y\\
=&\,\Big(\frac 1 2-K\Big)\frac{1}{2s(1-2s)}\big(-(h+\delta+\sigma)^{1-2s}+(\delta+\sigma)^{1-2s}-\delta^{1-2s}+(\delta+h)^{1-2s}\big)\\
&\,+\frac{1}{2}\frac{1}{2s(1-2s)}\big((\sigma-h)^{1-2s}-\sigma^{1-2s}+\delta^{1-2s}-(\delta-h)^{1-2s}\big)
\\
&\,+\frac{1}{(1-2s)(2-2s)}\big(-(\delta+h)^{2-2s}+(h+\delta+\sigma)^{2-2s}+\delta^{2-2s}-(\delta+\sigma)^{2-2s}\big)\\
&\,-\frac{1}{(1-2s)(2-2s)}\big(-(\sigma-h)^{2-2s}+(\delta-h)^{2-2s}+\sigma^{2-2s}-\delta^{2-2s}\big).
 \end{aligned}
\end{equation}
By Taylor expansion we have
\begin{eqnarray*}
&&-(h+\delta+\sigma)^{1-2s}+(\delta+\sigma)^{1-2s}-\delta^{1-2s}+(\delta+h)^{1-2s}\\
&=&-(1-2s)(\delta+\sigma)^{-2s}h+(1-2s)\delta^{-2s}h+\mathrm{O}(h^2);\\
&&(\sigma-h)^{1-2s}-\sigma^{1-2s}+\delta^{1-2s}-(\delta-h)^{1-2s}\\
&=&-(1-2s)\sigma^{-2s}h+(1-2s)\delta^{-2s}h+\mathrm{O}(h^2);\\
&&-(\delta+h)^{2-2s}+(h+\delta+\sigma)^{2-2s}+\delta^{2-2s}-(\delta+\sigma)^{2-2s}\\
&=&-(2-2s)\delta^{1-2s}h+(2-2s)(\delta+\sigma)^{1-2s}h+\mathrm{O}(h^2);\\
&&-(\sigma-h)^{2-2s}+(\delta-h)^{2-2s}+\sigma^{2-2s}-\delta^{2-2s}\\
&=&(2-2s)\sigma^{1-2s}h-(2-2s)\delta^{1-2s}h+\mathrm{O}(h^2).
\end{eqnarray*}
Plugging this in \eqref{termdue0} yields
\begin{equation}\label{termdue}
\begin{aligned}
E^s_{\delta,2}[h]
=&\,\Big(\frac 1 2-K\Big)\frac{h}{2s}\big(\delta^{-2s}-(\delta+\sigma)^{-2s}\big)+\frac 1 2\frac{h}{2s}\big(\delta^{-2s}-\sigma^{-2s}\big)\\
&\,+\frac{h}{1-2s}\big((\delta+\sigma)^{1-2s}-\sigma^{1-2s}\big)+\mathrm{O}(h^2)\\
\le &\,(1-K)\frac{h}{2s}\delta^{-2s}+ (K- 1 )\frac{h}{2s} \sigma^{-2s} 
+\frac{h}{1-2s}(\Lambda +\sigma)^{1-2s} +\mathrm{O}(h^2)\\
\le &\,(1-K)\frac{h}{2s}\delta^{-2s}+C(\Lambda,s,\sigma) h.
\end{aligned}
\end{equation}
By \eqref{termuno} and \eqref{termdue} we get \eqref{nuocla}, thus concluding the proof of the proposition.
\end{proof}
\begin{remark}
\rm{Clearly, under the same assumptions of 
 Proposition \ref{spostouna}, and by arguing verbatim  as in its  proof,  one also has 
 \begin{equation*}
-\nabla \E^s(X)\cdot e_{\lambda(\xi_1)}\le \frac{1-K}{2s}\delta^{-2s}+C(\Lambda,s,\sigma).
\end{equation*}
 }
\end{remark}
\subsection{The supercritical case $\frac 1 2\le s<1$}
Here we compute the first and second variations of the functional $\E^s_\ep$ defined in \eqref{enesupercrit}.
For every $X\in\rs^\Lambda$ and for every $\xi\in\sss(X)$, we recall that the measure $\mu^\xi$ is defined in \eqref{misurasing} and we set
$\mu^{\hat\xi}:=\mu^X-\mu^{\xi}$.
\begin{proposition}\label{prop:firstvarsuper}
Let $\frac 1 2\le s<1$, $0<\ep<\Lambda$ and let $X\in\rs^{\Lambda}$.
Let $\xi,\eta\in\sss(X)$ with $\xi,\eta>0$ and $\xi\neq \eta$.
Then
\begin{equation}\label{primasuper}
\partial_{x_{\lambda(\xi)}}\E_\ep^{s}(X) = 
(-\Delta)^s u_{\ep\ep}^X(\xi).
\end{equation}
Moreover,
\begin{equation}\label{secondasuper}
\partial^2_{x_{\lambda(\xi)}^2}\E_\ep^{s}(X)
=- (-\Delta)^s \mu_{\ep\ep}^{\hat{\xi}}(\xi).
\end{equation}
and
\begin{equation}\label{terzasuper}
 \partial^2_{x_{\lambda(\xi)},x_{\lambda(\eta)}}\E_\ep^{s}(X)
 =(-\Delta)^s \mu_{\ep\ep}^{\eta}(\xi).
\end{equation}

\end{proposition}
\begin{proof}
We proceed as in the proof of Proposition \ref{prop:firstvarsub}.
 Let $h\neq 0$ and let  $I_{h}$ be the (for instance open) interval with extreme points $\xi$ and $\xi+h$. We set $A_{h}:=I_{h}+\Lambda\Z$.
By construction, for $|h|$ small enough, $X+h e_{\lambda(\xi)}\in\X^\Lambda$
and
 $u^{X+h e_{\lambda(\xi)}} =u^X + \text{sgn}(h) \chi_{A_h}$.
To simplify notation we consider only the case $h>0$.
Moreover, we set $\chi_{A_h,\ep}:=\chi_{A_h}\ast\rho_\ep$. Then,
\begin{equation}\label{unosuper}
\begin{aligned}
&\,\E_\ep^s(X+he_{\lambda(\xi)})-\E_\ep^s(X)\\
=&\,\frac{1}{2}\int_{0}^{\Lambda}\ud x\int_{\R}\frac{(\chi_{A_{h},\ep}(x)-\chi_{A_h,\ep}(y))\big((2u_\ep^X(x)+\chi_{A_h,\ep}(x))-(2u_\ep^X(y)+\chi_{A_h,\ep}(y))\big)}{|x-y|^{1+2s}}\ud y\\
=&\,2\int_{0}^{\Lambda}\ud x\,\chi_{A_h,\ep}(x)\int_{\R}\frac{u^X_\ep(x)-u^X_\ep(y)}{|x-y|^{1+2s}}\ud y+\frac{1}{2}\int_{0}^{\Lambda}\ud x\int_{\R}\frac{(\chi_{A_{h},\ep}(x)-\chi_{A_h,\ep}(y))^2}{|x-y|^{1+2s}}\ud y\\
=&\,\int_{0}^{\Lambda}\chi_{A_h,\ep}(x)(-\Delta)^su_\ep^X(x)\ud x+\frac{1}{2}\int_{0}^{\Lambda}\ud x\int_{\R}\frac{(\chi_{A_{h},\ep}(x)-\chi_{A_h,\ep}(y))^2}{|x-y|^{1+2s}}\ud y,
\end{aligned}
\end{equation}
where the last equality follows by the very definition of fractional laplacian in \eqref{laplafra}, noticing that the function $u_\ep^X$ is $C^\infty$.

On the one hand, notice that $\chi_{A_h,\ep}$ coincides with the $\Lambda$-periodic extension of the function $\chi_{I_h}\ast\rho_\ep$ and that $\frac{1}{h}\chi_{I_h}\weakstar\delta_{\xi}$ as $h\to 0$, so that
\begin{equation}\label{duesuper}
\lim_{h\to 0}\frac{1}{h}\int_{0}^{\Lambda}\chi_{A_h,\ep}(x)(-\Delta)^su_\ep^X(x)\ud x=\int_{0}^\Lambda (\delta_{\xi}\ast\rho_\ep)(x)(-\Delta)^su_\ep^X(x)\ud x=(-\Delta)^su_{\ep\ep}^X(\xi).
\end{equation}
On the other hand, the function $\chi_{A_h,\ep}$ is Lipschitz continuous with Lipschitz constant $\|\chi_{A_h,\ep}\|_{W^{1,\infty}}\le C\frac{h}{\ep}$, whence we deduce that
\begin{equation}\label{tresuper}
\begin{aligned}
&\,\frac{1}{2}\int_{0}^{\Lambda}\ud x\int_{\R}\frac{(\chi_{A_{h},\ep}(x)-\chi_{A_h,\ep}(y))^2}{|x-y|^{1+2s}}\ud y\\
=&\,\frac{1}{2}\int_{0}^{\Lambda}\ud x\int_{0}^\Lambda\frac{(\chi_{A_{h},\ep}(x)-\chi_{A_h,\ep}(y))^2}{|x-y|^{1+2s}}\ud y+\,\sum_{\newatop{z\in\Z}{z\neq 0}}\frac{1}{2}\int_{0}^{\Lambda}\ud x\int_{z\Lambda}^{(z+1)\Lambda}\frac{(\chi_{A_{h},\ep}(x)-\chi_{A_h,\ep}(y))^2}{|x-y|^{1+2s}}\ud y\\
\le&\,C\frac{h^2}{\ep^2}\int_{0}^{\Lambda}\ud x\int_{0}^\Lambda|x-y|^{1-2s}\ud y+C\frac{h^2}{\ep^2}\sum_{\newatop{z\in\Z}{z\neq 0}}\int_{0}^{\Lambda}\ud x\int_{0}^{\Lambda}\frac{1}{|x-y+\Lambda z|^{1+2s}}\ud y\le C\frac{h^2}{\ep^2}.
\end{aligned}
\end{equation}
By combining \eqref{unosuper} with \eqref{duesuper} and \eqref{tresuper}, we get \eqref{primasuper}.\\
We now turn to \eqref{secondasuper}. The proof is similar to \eqref{seconda}. As before let $\hat e_{\lambda(\xi)}$ denotes the vector with the $\lambda(\xi)$-th entry equal to $0$ and all the remaining ones equal to $1$. Again, we consider only the case $h>0$ (small enough) and we set $\hat I_h:=\bigcup_{\eta\in\sss(X)\setminus\{\xi\}}(\eta-h,\eta)$ and
$\hat A_h:= \hat I_h+\Lambda\Z$.
Then, for $h$ small enough, $u^{X-h\hat e_{\lambda(\xi)}}_{\ep\ep}\equiv u^X_{\ep\ep}-(\chi_{\hat A_h})_{\ep\ep}$.
By \eqref{primasuper}, using the invariance by global translations, we thus find
\begin{multline*}
 \partial^2_{x_{\lambda(\xi)}^2}\E_\ep^{s}(X) =\lim_{h\to 0} \frac{(-\Delta)^s u_{\ep \ep}^{X+h e_{\lambda(\xi)}}(\xi+h)-(-\Delta)^s u_{\ep \ep}^{X}(\xi)}{h}\\
= \lim_{h\to 0}\frac{(-\Delta)^s u_{\ep \ep}^{X+h \hat e_{\lambda(\xi)}}(\xi)-(-\Delta)^s u_{\ep \ep}^{X}(\xi)}{h}=-\lim_{h\to 0} \frac{(-\Delta)^s (\chi_{\hat A_h})_{\ep\ep}(\xi)}{h}\\
 =-2\lim_{h\to 0} \frac{1}{h}\int_{\R} \frac{((\chi_{\hat A_h})_{\ep\ep}(\xi)-\chi_{\hat A_h})_{\ep\ep}(y)}{|y-\xi|^{1+2s}} \ud y=-2\int_{\R} \frac{\mu_{\ep\ep}^{\hat{\xi}}(\xi)-\mu_{\ep \ep}^{\hat{\xi}}(y)}{|y-\xi|^{1+2s}} \ud y=-(-\Delta)^s \mu_{\ep\ep}^{\hat\xi}(\xi).
\end{multline*}
We finally prove \eqref{terzasuper}.
Let again $h>0$ and set $A'_h:=(\eta,\eta+h)+\Lambda\Z$; then, for $h$ small enough, $ u^{X+h e_\lambda(\eta)}_{\ep\ep}= u^{X}_{\ep\ep}+ (\chi_{A'_h})_{\ep\ep}$.
Then, by \eqref{primasuper} we obtain
\begin{equation*}
\begin{aligned}
 {\partial^2_{x_{\lambda(\xi)},x_{\lambda(\eta)}}}\E^{s}(X) =&\, \lim_{h\to 0} \frac{(-\Delta)^s u^{X+h e_\lambda(\eta)}_{\ep\ep}(\xi)- (-\Delta)^s u^{X}_{\ep\ep}(\xi)}{h}\\
 =&\, \lim_{h\to 0} \frac{(-\Delta)^s (\chi_{A'_h})_{\ep\ep}(\xi)}{h}
=(-\Delta)^s\mu^\eta_{\ep\ep}(\xi).
\end{aligned}
\end{equation*}
\end{proof}
\begin{remark}
\rm{ Let us notice that sending $\ep\to0$, formulas \eqref{primasuper}, \eqref{secondasuper} and \eqref{terzasuper} converge to \eqref{prima}, \eqref{seconda} and \eqref{terza}.}
\end{remark}
To any configuration $X\in\X^\Lambda$ we associate the minimal distance between its entries 
\begin{equation}\label{dmax}
\delta(X):=\min_{\newatop{\lambda,\lambda'\in\{1,\ldots,\Lambda\}}{\lambda\neq \lambda'}}\nper{x_\lambda-x_{\lambda'}},
\end{equation}
where the norm $\nper{\cdot}$ is defined in \eqref{nperdef}. We notice that $\delta(X)= \min_{i\neq j} |\mathsf x_i - \mathsf x_j|$ and that 
$\sup_{X\in\X^\Lambda}\delta(X) =  1$.
\begin{proposition}\label{varscepossuper}
Let $\frac 1 2\le s < 1$, $\Lambda\in\N$ and $\ep>0$. Then, for every $X\in\rs^\Lambda$ and for every $h\in\R^\Lambda$ we have
\begin{equation}\label{posvarsecsuper}
\langle\nabla^2\E_\ep^s(X)h,h\rangle=\sum_{\newatop{\xi,\eta\in\sss(X)}{\xi\neq\eta}}a_\ep^s(\xi,\eta)(h_{\lambda(\xi)}-h_{\lambda(\eta)})^2,
\end{equation}
where
\begin{equation*}
a^s_\ep(\xi,\eta):=2\int_{\R}\frac{\mu_{\ep\ep}^\eta(y)-\mu_{\ep\ep}^\eta(\xi)}{|y-\xi|^{1+2s}}\ud y\qquad\textrm{for every }\xi,\eta\in \sss(X) \textrm{ with }\xi\neq\eta.
\end{equation*}
Assume now $\delta(X)\ge 4\ep$; then the following facts hold true.


If $h\cdot e=0$, then
\begin{equation}\label{stimasuperpeggio}
\langle\nabla^2\E_\ep^s(X)h,h\rangle\ge \Gamma(s) \Lambda^{-2s}|h|^2,
\end{equation}
for some positive constant $\Gamma(s)$ depending only on $s$.

Moreover, for every $X\in\cs^{\Lambda}$, then
\begin{equation}\label{stimasupermeglio}
\langle\nabla^2\E_\ep^s(X)h,h\rangle\le \gamma(s)\Lambda^{-2s}|h|^2\qquad\textrm{for every }h\in\R^{\Lambda}\textrm{ with }h\cdot e=0,
\end{equation}
for some positive constant $\gamma(s)$ depending only on $s$.
\end{proposition}
\begin{proof}
Identity \eqref{posvarsecsuper} is an immediate consequence of Proposition \ref{prop:firstvarsuper}.
Now, since $\delta(X)\ge 4\ep$, we have that $\mu_{\ep\ep}^\eta(\xi)=0$ for every $\xi,\eta\in\sss(X)$ with $\xi\neq\eta$ so that 
\begin{equation}\label{utile}
a_\ep^s(\xi,\eta)=2\int_{\R}\frac{\mu_{\ep\ep}^\eta(y)}{|y-\xi|^{1+2s}}\ud y= 2\sum_{z\in\Z}\int_{B_{2\ep}(\eta)}\frac{\big(\rho_\ep\ast(\rho_\ep\ast\delta_\eta)\big)(y)}{|y-\xi+\Lambda z|^{1+2s}}\ud y.
\end{equation}
Let $\xi,\eta\in\sss(X)$ with $\xi\neq\eta$ and
note that for every $y\in B_{2\ep}(\eta)$
\begin{equation*}
\frac{2^{-1-2s}}{|\eta-\xi+\Lambda z|^{1+2s}}\le \frac{1}{|y-\xi+\Lambda z|^{1+2s}}\le \frac{2^{1+2s}}{|\eta-\xi+\Lambda z|^{1+2s}},
\end{equation*}
so that, by \eqref{utile},
\begin{equation}\label{comportabene}
2^{-1-2s}a^s(\xi,\eta)\le a_\ep^s(\xi,\eta)\le 2^{1+2s}a^s(\xi,\eta),
\end{equation}
where the coefficients
$a^s(\xi,\eta)$ are defined in \eqref{coeffa}.
As a consequence, in view of Remark \ref{subcomesuper} and of \eqref{comportabene}, the estimates \eqref{darichiamareinsuper} and \eqref{darichiamareinsuper2} yield \eqref{stimasuperpeggio} and  \eqref{stimasupermeglio}, respectively.
This concludes the proof.
\end{proof}

\section{Minimal configurations}
By symmetry arguments it is straightforward to check that the configurations in $\cs^\Lambda$ (see  \eqref{crit}) are critical points of the energy functionals; more precisely, for all $X\in \cs^\Lambda$
\begin{equation}\label{laplafranullo}
(-\Delta)^s u^X(\xi)=0, \quad (-\Delta)^s\bar u^X_{\ep\ep}(\xi)=0 \qquad \textrm{ for any } \xi \in \sss(X).
\end{equation}

The next theorem establishes that  for $s<\frac 12$,   the ground states of $\E^s$ coincide with $\cs^\Lambda$.
\begin{theorem}\label{minsubcrit}
Let $0<s<\frac 1 2$ and $\Lambda\in\N$. Then, the set of critical points of the energy $\E^s$ in $\rs^\Lambda$ is given by
$\cs^\Lambda$. Moreover, $\cs^\Lambda$ coincides with  the set of all  local and global minimizers of $\E^s$ in $\X^\Lambda$.
\end{theorem}
\begin{proof}
Let $X$ be a local minimizer of $\E^s$ in $\X^\Lambda$; by Lemma \ref{missinglemma2} we have immediately that $X\in\rs^\Lambda$, so that it is a critical point of the energy. Since, in view of \eqref{stimasubpeggio}, the second variation of $\E^s$ is strictly positive modulo rigid translations, there is at most one - up to rigid translations - critical point of the energy $\E^s$. This fact, together with    \eqref{prima} and \eqref{laplafranullo} implies that the set of critical points coincides with $\cs^\Lambda$. In turn $X\in \cs^\Lambda$ and, since $\E^s$ is constant on $\cs^\Lambda$,  $X$ is a global minimizer.
\end{proof}
In the next theorem we consider  the critical cases $s\ge \frac 1 2$.
\begin{theorem}\label{mincrit}
Let $\frac 1 2 \le s<1$ and $\Lambda\in\N$. There exists $\bar \ep>0$ such that for every $0<\ep<\bar\ep$ the set of global minimizers of $\E_\ep^s$ is given by $\cs^\Lambda$.
\end{theorem}
Before proving Theorem \ref{mincrit} we state and prove the following result, providing an asymptotic ``expansion'' of the minimal energy $\E^s_\ep$ as $\ep\to 0$.
To this end, we recall the definition of $\delta(X)$ in \eqref{dmax} and, for every $\frac 1 2\le s<1$, $\delta>0$, we set
\begin{equation}\label{costosigma}
\sigma^s(\delta):=\left\{\begin{array}{ll}
|\log \delta|&\textrm{if }s=\frac 1 2\\
&\\
\frac{2^{1-2s}}{2s(2s-1)}\delta^{1-2s}&\textrm{if }\frac 1 2<s<1.
\end{array}\right.
\end{equation}
\begin{proposition}\label{quasigamma}
Let $\frac 1 2\le s<1$ and $\Lambda\in\N$.
Then, for $\delta$ small enough (depending on $\Lambda$ and $s$),
\begin{equation}\label{theclaim}
\liminf_{\ep\to 0}\Big(\inf_{X\in\X^\Lambda:\delta(X)\le \delta}\E_\ep^s(X)-\Lambda\sigma^s(\ep)\Big)\ge C(s) \sigma^s(\delta),
\end{equation}
for some positive constant $C(s)$ depending only on $s$. 
\end{proposition}
\begin{proof}
Let $X\in\X^\Lambda$ and let $\sss(X):=\{\xi_1,\ldots,\xi_K\}$ with $1\le K\le \Lambda$ and $\xi_k<\xi_{k+1}$ for $k=1,\ldots,K-1$. 
Since the maximal distance between two consecutive particles is larger than (or equal to) $1$, we can assume without loss of generality that $\xi_1\ge \frac 1 2$ and $\xi_K\le \Lambda-\frac 1 2$; 
let 
$$
0<\ep < r< \frac12 \min\Big\{ 1,\min_{k=1,\ldots,K-1}|\xi_{k+1}-\xi_{k}|\Big\}. 
$$
Moreover, we can assume without loss of generality that $X=X^\ord$. 
 
 Then
\begin{equation}\label{enrisc}
\begin{aligned}
\E^{s}_\ep(X)\ge&\,\frac 1 2\sum_{k=1}^{K}\int_{A_{\ep,r}(\xi_k)}\ud x\int_{A_{\ep,r}(\xi_k)}\frac{|u_\ep^X(x)-u_\ep^{X}(y)|^2}{|x-y|^{1+2s}}\ud y\\
&\,+\frac 1 2\int_{(0,\Lambda)\setminus\bigcup_{k=1}^{K} \bar B_r(\xi_k)}\ud x\int_{\R\setminus\bigcup_{k=1}^{K}\bar B_r(\xi_k)}\frac{|u_\ep^X(x)-u_\ep^{X}(y)|^2}{|x-y|^{1+2s}}\ud y\\
&=: I_{\ep,r,\mathrm{sr}}^s+I_{\ep,r,\mathrm{lr}}^s.
\end{aligned}
\end{equation}
Now we estimate $I^s_{\ep,r,\mathrm{sr}}$ and $I^s_{\ep,r,\mathrm{lr}}$.
 
To this end, we preliminarily notice that for $x,y\in \R\setminus\bigcup_{k=1}^K \bar B_\ep(\xi_k)$ it holds
\begin{equation}\label{imp}
u_\ep^X(x)-u_\ep^X(y)=u^X(x)-u^X(y)=x-y+\mu^X((x,y)),
\end{equation}
As a consequence
\begin{equation}\label{meglioI1}
\begin{aligned}
I_{\ep,r,\mathrm{sr}}^s=&\,\sum_{k=1}^{K}\Bigg(\frac 1 2\int_{A_{\ep,r}(\xi_k)}\ud x\int_{A_{\ep,r}(\xi_k)}|x-y|^{1-2s}\ud y\\
&\,-2 \mm(\xi_k)\int_{\xi_k-r}^{\xi_k-\ep}\ud x\int_{\xi_k+\ep}^{\xi_k+r}  (y-x)^{-2s}\ud y\\
&\,
+\mm^2(\xi_k) \int_{\xi_k-r}^{\xi_k-\ep}\ud x\int_{\xi_k+\ep}^{\xi_k+r} (y-x)^{-1-2s}\ud y\Bigg)\\
\ge&\,\sum_{k=1}^{K}\Bigg(-2 \mm(\xi_k)\int_{\xi_k-r}^{\xi_k-\ep}\ud x\int_{\xi_k+\ep}^{\xi_k+r}  (y-x)^{-2s}\ud y\\
&\,+\mm^2(\xi_k) \int_{\xi_k-r}^{\xi_k-\ep}\ud x\int_{\xi_k+\ep}^{\xi_k+r} (y-x)^{-1-2s}\ud y\Bigg).
\end{aligned}
\end{equation}
We first discuss the case $s>\frac 1 2$.

By straightforward computations, we have that
\begin{equation}\label{contiinu}
\lim_{\e\to 0 } \int_{-r}^{-\ep}\ud x\int_{\ep}^{r}{(y-x)^{-2s}}\ud y=\frac{2-2^{2-2s}}{(2s-1)(2-2s)} r^{2-2s}\le C(s),
\end{equation}
and, by Taylor expansion,
\begin{equation}\label{unicocontosuper}
\begin{aligned}
\int_{-r}^{-\ep}\ud x\int_{\ep}^r(y-x)^{-1-2s}\ud y=&\,\frac{1}{2s(2s-1)}\big((2\ep)^{1-2s}+(2r)^{1-2s}-2(\ep+r)^{1-2s}\big)\\
=&\,
\sigma^s(\ep)+\frac{2^{1-2s}-2}{2s(2s-1)}r^{1-2s}+\mathrm{O}_r(\ep),
\end{aligned}
\end{equation}
where $\mathrm{O}_r:\R^+\to \R$ satisfies, for fixed $r>0$,  $\limsup_{\ep\to 0}\frac{\mathrm{O}_r(\ep)}{\ep} =:C_r < + \infty$.
By \eqref{enrisc}, \eqref{meglioI1}, \eqref{contiinu} and \eqref{unicocontosuper},  using also that $I_{\ep,r,\mathrm{lr}}^s$ is non-negative, for $\ep$ small enough, we get
\begin{equation}\label{sifastimadopo}
\begin{aligned}
\E^s_\ep(X)\ge&\, \sum_{k=1}^K\widehat{m}^2(\xi_k)\Big(\sigma^s(\ep)+ \frac{2-2^{1-2s}}{2s(2s-1)}r^{1-2s}+\mathrm{O}_r(\ep)\Big)-2\Lambda C(s)\\
\ge&\, \Lambda\sigma^s(\ep)+\Lambda\frac{1}{2s(2s-1)}r^{1-2s}+\mathrm{O}_r(\ep)-2\Lambda C(s),
\end{aligned}
\end{equation}
where the last inequality follows using that $1-2s<0$ and that
\begin{equation}\label{megliosing}
\sum_{k=1}^{K}\mm^2(\xi_k)\ge \Lambda.
\end{equation}
If $\delta(X)=0$, then the inequality \eqref{megliosing} is strict, whence \eqref{theclaim} immediately follows.

Assume now that $\delta(X)>0$; since $r<\delta(X)$,
taking the $\liminf$ as $\ep\to 0$ in \eqref{sifastimadopo}, we have
\begin{equation*}
\liminf_{\ep\to 0}\Big(\inf_{X\in\X^\Lambda:\delta(X)\le \delta}\E^s_\ep(X)-\Lambda\sigma^s(\ep)\Big)\ge \Lambda\frac{1}{2s(2s-1)}r^{1-2s}-2\Lambda C(s).
\end{equation*}
This concludes the proof of \eqref{theclaim} for $s>\frac 1 2$.
\medskip

Now, we discuss the case $s=\frac 1 2$.

To this end, we preliminarily observe that for every $0<\gamma<\bar\gamma$ it holds
\begin{equation*}
\int_{-\bar\gamma}^{-\gamma}\ud x\int_{\gamma}^{\bar\gamma}(y-x)^{-2}\ud y=2 \log(\gamma+\bar\gamma)-\log(2\gamma)-\log(2\bar\gamma)=\log\frac{\gamma+\bar\gamma}{\gamma}+\log\frac{\gamma+\bar\gamma}{\bar\gamma}-2\log 2,
\end{equation*}
thus yielding
 \begin{equation}\label{unicoconto+}
\lim_{\frac{\gamma}{\bar\gamma}\to 0}\Big(\int_{-\bar\gamma}^{-\gamma}\ud x\int_{\gamma}^{\bar\gamma}(y-x)^{-2}\ud y-\log\frac{\bar\gamma}{\gamma}\Big)=-2\log 2.
\end{equation}
Analogously, one can check that
\begin{equation}\label{unicoconto2}
\int_{-\bar\gamma}^{-\gamma}\ud x\int_{\gamma}^{\bar\gamma}(y-x)^{-1}\ud y=2\gamma(1+\log(2\gamma))-2\bar\gamma(1-\log(2\bar\gamma))-2(\gamma+\bar\gamma)\log(\gamma+\bar\gamma).
\end{equation}
Therefore, by combining \eqref{meglioI1} together with \eqref{unicoconto+}, \eqref{unicoconto2} (applied with $\gamma=\ep$ and $\bar\gamma=r$) and with \eqref{megliosing}, for 
$r$ fixed
 and $\ep$ small enough we deduce that
\begin{equation}\label{estI1}
I_{\ep,r,\mathrm{sr}}^{\frac 1 2}\ge\sum_{k=1}^{K}\mm^2(\xi_k)\log\frac{r}{\ep}-C_r(\ep)\ge  \Lambda\log\frac{r}{\ep}-C_r(\ep),
\end{equation}
where $C_r(\ep)$ satisfies $\limsup_{\ep\to 0}C_r(\ep)=:C_r<+\infty$.
We recall that the inequality in \eqref{megliosing} is an equality if and only if $\mm(\xi_k)=1$ for
every $k=1,\ldots,K$, namely, if $X\in\rs^\Lambda$.
Therefore, since $I_{\ep,r,\mathrm{lr}}^{\frac 1 2}$ is non-negative, we have that the claim is trivially satisfied if $X\in\X^\Lambda\setminus\rs^\Lambda$.

We focus now on the case $X\in\rs^\Lambda$.
Notice that, in view of \eqref{imp}, the term $I_{\ep,r,\mathrm{lr}}^{\frac 12}$ does not depend on $\ep$, so that we can write $I_{\ep,r,\mathrm{lr}}^{\frac 1 2}=:I_{r,\mathrm{lr}}^{\frac 1 2}$, and
by \eqref{enrisc} and \eqref{estI1}, we get that for $\ep$ small enough,
\begin{equation}\label{quasiclaim}
\E_\ep^{\frac 1 2}(X)-\Lambda|\log\ep|\ge I_{r,\mathrm{lr}}^{\frac 1 2}+\Lambda\log r-C_r(\ep)\qquad\textrm{ for every }X\in\rs^\Lambda\textrm{ with }\delta(X)\ge 2r.
\end{equation}
We can assume without loss of generality that $\delta(X)\le 1$.
In order to prove the claim it is enough to study the behavior of $I_{r,\mathrm{lr}}^{\frac 1 2}$ as $r\to 0$.
Let $\gamma_r:(0,1)\to \N$ be the function that at any $p\in (0,1)$ associates the number $\gamma_r(p)$ of connected components of the set $\bigcup_{\lambda=1}^{\Lambda} B_{r^p}(x_\lambda)\cup B_{r^p}(x_1+\Lambda)$.
Trivially, the function $\gamma_r$ is piecewise constant; we denote by $0<p_1<\ldots<p_L<1$ the jump points of $\gamma_r$ and we set $p_0:=0$ and $p_{L+1}:=1$. Let $0<\eta<\frac 1 2\min_{l=1,\ldots,L-1}(p_{l+1}-p_l)$. Then there exists $J_{r}^l$ ``annuli'' $A_{r^{p_l-\eta},r^{p_{l-1}+\eta}}(\zeta_j)$, with $j=1,\ldots,J_{r}^l$, that are pairwise disjoint for every $j$ and $l$ and satisfy
$$
\bigcup_{\lambda=1}^{\Lambda} B_r(x_\lambda)\subset\bigcup_{j}B_{r^{p_l-\eta}}(\zeta_j).
$$
We highlight that, $\mu^X((x,y))=\mu^X(B_{r^{p_k-\eta}}(\zeta_j))=:[\mu_j]$ whenever $\zeta_j-r^{p_{l-1}+\eta}<x<\zeta_j-r^{p_l-\eta}$ and $\zeta_j+r^{p_{l}-\eta}<y<\zeta_j+r^{p_{l-1}+\eta}$.
Moreover, by construction,
 for every $l$ there exists (at least one) $j=1,\ldots,J^l_r$ such that $|[\mu_j]|\ge 2$  and hence
$|[\mu_j]|^2\ge |[\mu_j]|+3$.
Therefore, by \eqref{imp}, using \eqref{unicoconto+} and \eqref{unicoconto2},
arguing as in \eqref{sifastimadopo}, for $r$ small enough
we have that
\begin{equation*}
\begin{aligned}
I_{r,\mathrm{lr}}^{\frac 1 2}\ge&\, \frac 1 2\sum_{l=1}^{L+1}\sum_{j=1}^{J_r^l}\int_{A_{r^{p_l-\eta},r^{p_{l-1}+\eta}}(\zeta_j)}\ud x\int_{A_{r^{p_l-\eta},r^{p_{l-1}+\eta}}(\zeta_j)}\frac{|\mu^X(x,y)|^2}{|x-y|^2}\ud y-C\\
\ge&\, \sum_{l=1}^{L+1}\sum_{j=1}^{J_r^l}|[\mu_j]|\int_{\zeta_j-r^{p_{k-1}+\eta}}^{\zeta_j-{r^{p_k-\eta}}}\ud x\int_{\zeta_j+r^{p_k-\eta}}^{\zeta_j+r^{p_{k-1}+\eta}}\frac{1}{|x-y|^2}\ud y\\
&\,+3 \sum_{k=1}^{K+1}\int_{\zeta_j-r^{p_{l-1}+\eta}}^{\zeta_j-{r^{p_l-\eta}}}\ud x\int_{\zeta_j+r^{p_l-\eta}}^{\zeta_j+r^{p_{l-1}+\eta}}\frac{1}{|x-y|^2}\ud y-C\\
\ge&\,
(\Lambda+3)\sum_{l=1}^{K+1}(p_{l}-p_{l-1}-2\eta)|\log r|-C
=(\Lambda+3)(1-2(K+1)\eta)|\log r|-C,
\end{aligned}
\end{equation*}
where in the last inequality we have used
 that $\sum_{j=1}^{J_r^k}[\mu_j]=\mu^X([0,\Lambda))=\Lambda$.
 
 By the arbitrariness of $\eta$ we have that $I^{\frac 1 2}_{r,\mathrm{lr}}\ge (\Lambda+3)|\log r|-C$, which, in view of \eqref{quasiclaim}, for $r$ small enough, implies
 $$
 \E_\ep^{\frac 1 2}(X)-\Lambda|\log\ep|\ge 3|\log r|-C,
 $$
whence \eqref{theclaim} follows also in this case.
This concludes the proof of the whole result.
\end{proof}
\begin{proof}[Proof of Theorem \ref{mincrit}]
By straightforward computations, for every configuration $\overline X\in\cs^{\Lambda}$ we have
\begin{equation}\label{base}
\E_\ep^{s}(\overline X)\le \Lambda\sigma^s(\ep)+\bar C(s).
\end{equation}
Let $\omega^s(\cdot):=C(s)\sigma^s(\cdot)$ be the quantity in the right-hand side of \eqref{theclaim}  and let
 $\bar\delta=\bar\delta(s)>0$ be such that $\frac {\omega^s(\delta)}{2}>\bar C(s)$ for every $0<\delta<\bar \delta$.
By Proposition \ref{quasigamma}, there exists $\bar\ep=\bar\ep(\bar\delta)>0$ such that for every 
$0<\ep<\bar \ep$
it holds
\begin{equation}\label{daprop}
\inf_{X\in\X^\Lambda:\delta(X)<\bar\delta}\E_\ep^s(X)\ge \Lambda\sigma^s(\ep)+\frac{\omega^s(\delta)}{2}.
\end{equation}
Let 
\begin{equation}\label{eppiccolo}
0<\ep<\min\Big\{\bar \ep,\frac{\bar\delta}{4}\Big\}.
\end{equation}
By \eqref{base} and \eqref{daprop}, we have that every minimizer $X$ of $\E^s_\ep$ in $\X^\Lambda$ satisfies
\begin{equation}\label{dalba}
\delta(X)>\bar\delta,
\end{equation}
and hence, in particular, 
\begin{equation}\label{limitedist}
\delta(X)> 4\ep.
\end{equation}
Let $\rs^\Lambda_{\bar\delta}$ denote the class of configurations satisfying \eqref{dalba}; notice that such a set is convex if the configurations are described in terms of the distances between nearest neighboring particles.
By \eqref{limitedist}, we can use \eqref{stimasuperpeggio} to deduce that $\nabla^2\E_\ep^s$ is strictly positive modulo rigid translations on $\rs^\Lambda_{\bar\delta}$, whence we have immediately that there is at most one - up to rigid translations - critical point of the energy $\E_\ep^s$ in $\rs^\Lambda_{\bar\delta}$.
This fact, together with \eqref{primasuper} and \eqref{laplafranullo}, yields that  the set of critical points of $\E_\ep^s$ in $\rs^\Lambda_{\bar\delta}$ coincides with 
$\cs^\Lambda$.
Now, since $\E_\ep^s$ is constant on $\cs^\Lambda$,  we get that $\cs^\Lambda$ is the set of  global minimizers of $\E_\ep^s$.

\end{proof}
\section{Dynamics}
In this section we study the gradient flows of the energy functionals $\E^s$ and $\E_\ep^s$.

\subsection{The subcritical case $0<s<\frac 1 2$}\label{subsec:dynsub}
Here we analyze the gradient flow system of $\E^s$ starting from a regular datum $X^0\in\rs^\Lambda$. With a little abuse of notation we say that the map $t\mapsto X(t)\in \rs^\Lambda$ is absolutely continuous on some interval $I$, and  we write $X\in AC(I)$, if the maps $t\mapsto \mathsf x_z(t)$ are absolutely continuous in $I$ (for all $z\in \Z$).
Notice that, if $X\in AC(I)$, the quantities $\dot X(t)$ and $\nabla \E^s(X(t))$ are well defined (in the obvious way)  for almost all $t\in I$.  
In this sense we understand the Cauchy problem we aim at studying
\begin{equation}\label{cauchy}
\left\{
\begin{array}{l}
\dot{X}(t)=-\nabla\E^s(X(t))\\
X(0)=X^0.
\end{array}
\right.
\end{equation}
Notice that
if we denote by $\nabla \E^s(\XX)$ the sequence obtained by setting $\partial_{x_{\lambda+\Lambda z}}\E^s(\XX):=\partial_{x_\lambda}\E^s(X)$ for every $\lambda=1,\ldots,\Lambda$ and for every $z\in\Z$, then the Cauchy problem above is equivalent to 
\begin{equation}\label{cauchyequiv}
\left\{
\begin{array}{l}
\dot{\XX}(t)=-\nabla\E^s(\XX(t))\\
\XX(0)=\XX^0.
\end{array}
\right.
\end{equation}
In view of Proposition \ref{prop:firstvarsub} we have that there exists $T>0$ such that the problem \eqref{cauchy} admits a unique $C^1$ (actually, smooth) solution in $[0,T)$.


We start by proving that the problem \eqref{cauchy} has (unique and $C^1$) solution in $[0,+\infty)$ and that the configuration $X(t)$ is regular for every $t\ge 0$.
\begin{proposition}\label{nocollis<}
Let $0<s<\frac 1 2$.
Let $X^0\in\rs^\Lambda$ and let $T_{\max}$ be the maximal existence time for the Cauchy problem \eqref{cauchy}.
Then $T_{\max}=+\infty$.
\end{proposition}
In order to prove Proposition \ref{nocollis<}, we need an auxiliary result and some further notation.
For every $X,\, Y\in \X^\Lambda$ we set
\begin{equation}\label{distanza}
\dist(X,Y):= \Big(\sum_{\lambda=1}^\Lambda \nper{x_\lambda^\ord - y_\lambda^\ord}^2 \Big)^{\frac 12}.
\end{equation}
Let $\{X^n\}_{n\in\N}\subset\rs^\Lambda$ and $X\in\X^\Lambda$, be such that
\begin{equation}\label{converge}
\dist(X^n,X)\to 0\qquad\textrm{(as $n\to +\infty$)}.
\end{equation}
Assume that  $X^n\equiv (X^n)^\ord$, $X\equiv X^\ord$ and $\sss(X)\subset (0,\Lambda)$.
For every $\xi\in\sss(X)$, we recall that $x_{\phi(\xi)}$ is the first entry of $X$ which is equal to $\xi$, so that  $x^n_{\phi(\xi)},\ldots, x^n_{\phi(\xi)+\mm(\xi)-1}$ converge to $\xi$ as $n\to+\infty$;
for every $n\in\N$ we define
$$
\xi^n:=\mathrm{argmax}\{|t-\xi|\,:\,t\in\{x^n_{\phi(\xi)},x^n_{\phi(\xi)+\mm(\xi)-1}\}\}.
$$
For every $n\in\N$ we define the configuration 
 $\widehat X^n\in\X^\Lambda$ such that 
 \begin{equation}\label{hatdefi}
 \hat x^n_\lambda:=\xi^n\qquad\textrm{ if }\phi(\xi)\le \lambda\le \phi(\xi)+\mm(\xi)-1\textrm{ for some }\xi\in\sss(X). 
 \end{equation}
 We have, $\hat x^n_{\phi(\xi)}=\xi^n=x^n_{\phi(\xi)}$ whenever $\mm(\xi)=1$.
 It is easy to check that
 \begin{equation}\label{crucial}
\frac{1}{\sqrt{\Lambda}}\dist(\widehat X^n,X)\le \dist(X^n,X)\le\dist(\widehat X^n,X).
 \end{equation}
Furthermore, we define
\begin{equation}\label{distanzastramba}
\Delta_{X}(X^n):=\sum_{\lambda: x_\lambda=x_{\lambda+1}}(x_{\lambda+1}^n-x_{\lambda}^n)=\sum_{\xi\in\sss(X)}(x^n_{\phi(\xi)+\mm(\xi)-1}-x^n_{\phi(\xi)}).
\end{equation}
Finally, for general $X^n$ and $X$ (not necessarily coinciding with $(X^n)^\ord$ and $X^\ord$, respectively), we set $\Delta_{X}(X^n):=\Delta_{X^{\ord}}((X^n)^\ord)$.
\begin{lemma}\label{collision}
Let $\{X^n\}_{n\in\N}\subset\rs^\Lambda$ and $X\in\X^\Lambda\setminus\rs^\Lambda$, be such that   
\begin{equation}\label{convergeinlemma}
\dist(X^n,X)\to 0\qquad\textrm{(as $n\to +\infty$)}.
\end{equation}
Then,
\begin{equation}\label{infinito}
\lim_{n\to +\infty}\frac{\E^s(\widehat X^n)-\E^s(X^n)}{\Delta_X(X^n)}=+\infty.
\end{equation}
\end{lemma}
\begin{proof}
We can assume without loss of generality that $X=X^\ord$ and that $X^n= (X^{n})^\ord$ for every $n\in\N$. Moreover, in virtue of the translational invariance, we can assume that $x_1>0$ and that $x^n_1>0$ for every $n\in\N$.  

Let $\mathcal{M}(X):=\{\xi\in\sss(X)\,:\,\mm(\xi)\ge 2\}=\{\xi_1,\ldots,\xi_J\}$ with $\xi_j<\xi_{j+1}$. 
We will assume, just to fix the notation, that $\xi^n_{j}\equiv x^{n}_{\phi(\xi_j)+\widehat m(\xi_j)-1}$ for every $j=1,\ldots,J$ and for every $n\in\N$.

Let $K:=\sum_{j=1}^J(\mm(\xi_j)-1)$.
For every $k=0,1,\ldots,K-1$  
there are uniquely determined integers $1\le \bar\jmath\le J$ and $0\le l\le \mm(\xi_{\bar\jmath})-2$ such that
\begin{equation}\label{kform}
k=\sum_{j=1}^{\bar\jmath-1}(\mm(\xi_{j})-1)+l,
\end{equation} 
where, for $\bar\jmath=1$ the sum above is understood to be equal to zero. 
We define $X^n[k]=(x^n_1[k],\ldots,x^n_{\Lambda}[k])$, where
\begin{equation*}
x^n_{\lambda}[k]:=\left\{\begin{array}{ll}
x^n_{\phi(\xi_j)+\mm(\xi_j)-1}&\textrm{if }\phi(\xi_{ j})\le \lambda\le\phi(\xi_{j})+\mm(\xi_j)-1, \,\textrm{for }1\le j\le\bar\jmath-1,\\
x^n_{\phi(\xi_{\bar\jmath})+l}&\textrm{if }\phi(\xi_{\bar\jmath})\le \lambda\le\phi(\xi_{\bar\jmath})+l-1,\\
x^n_{\lambda}&\textrm{elsewhere}.
\end{array}
\right.
\end{equation*}
By definition, $X^n[0]\equiv X^n$.
Furthermore, we define $X^n[K]$ as
\begin{equation*}
x^n_{\lambda}[K]:=\left\{\begin{array}{ll}
x^n_{\phi(\xi_j)+\mm(\xi_j)-1}&\textrm{if }\phi(\xi_{ j})\le \lambda\le\phi(\xi_{j})+\mm(\xi_j)-1, \,\textrm{for }1\le j\le J,\\
x^n_{\lambda}&\textrm{elsewhere},
\end{array}
\right.
\end{equation*}
so that $X^n[K]\equiv \widehat X^n$.

We claim that, for $k$ as in \eqref{kform} with $l\ge 1$,
\begin{multline}\label{claim}
\E^s(X^n[k])-\E^s(X^n[k-1])\\
\ge\frac{1}{2s(1-2s)} (x^n_{\phi(\xi_{\bar\jmath})+l}-x^n_{\phi(\xi_{\bar\jmath})+l-1})^{1-2s}
-C(s,\Lambda)(x^n_{\phi(\xi_{\bar\jmath})+l}-x^n_{\phi(\xi_{\bar\jmath})+l-1}),
\end{multline}
and that, if $k=\sum_{j=1}^{\bar\jmath-1}(\mm(\xi_{j})-1)$ for some $\bar\jmath=1,\ldots, J+1$ (i.e., in the case $l=0$ in \eqref{kform} and $k=K$),
then
\begin{multline}\label{secoclaim}
\E^s(X^n[k])-\E^s(X^n[k-1])\ge\\
 \frac{1}{2s(1-2s)} (x^n_{\phi(\xi_{\bar\jmath})+\mm(\xi_{\bar\jmath})-1}-x^n_{\phi(\xi_{\bar\jmath})+\mm(\xi_{\bar\jmath})-2})^{1-2s}
-C(s,\Lambda)(x^n_{\phi(\xi_{\bar\jmath})+\mm(\xi_{\bar\jmath})-1}-x^n_{\phi(\xi_{\bar\jmath})+\mm(\xi_{\bar\jmath})-2}).
\end{multline}
We will prove only \eqref{claim}, being the proof of \eqref{secoclaim} the same up to notational changes.

For every $n\in\N$ 
we set
$$
I^n_l:=(x^n_{\phi(\xi_{\bar\jmath})+l-1},x^n_{\phi(\xi_{\bar\jmath})+l}),\qquad\qquad A^n_l:=I^n_l+\Lambda\Z.
$$
For every $k$, we set $u^n[k]:=u^{X^n[k]}$ and we notice that
$$
u^n[k]=u^n[k-1]+l
\chi_{A^n_l}.
$$
By arguing as in \eqref{fvar1}, one can  check that
\begin{equation}\label{differene}
\begin{aligned}
&\E^s(X^n[k])-\E^s(X^n[k-1])\\
=&\,2l\int_{I^n_l}\ud x\int_{\R}\frac{(u^{n}[k-1](x)+\frac {l}{2})-(u^n[k-1](y)+\frac {l}{2}\chi_{A^n_l}(y))}{|x-y|^{1+2s}}\ud y.
\end{aligned}
\end{equation}
By arguing as in \eqref{vienezero}, we get
\begin{equation}\label{vienezero2}
\int_{I^n_l}\ud x \int_{I^n_l}\frac{u^{n}[k-1](x)+\frac {l}{2}-\big(u^n[k-1](y)+\frac {l} 2\chi_{A^n_l}(y)\big)}{|x-y|^{1+2s}}\ud y=0.
\end{equation}
Let now $\rho>0$ be such that the intervals $B_{\rho}(\xi)$, with $\xi\in\sss(X)$ are pairwise disjoint and contained in $(0,\Lambda)$.

By construction, for every $x,y\in\R$
\begin{equation}\label{inftyest}
\Big|u^{n}[k-1](x)+\frac {l}{2}-\big(u^n[k-1](y)+\frac {l} 2\chi_{A^n_l}(y)\big)\Big|\le 2\Lambda;
\end{equation}
moreover, for $n$ large enough, we have that 
\begin{equation}\label{denoest}
|x-y|\ge \frac{1}{2}|\xi_{\bar\jmath}-y|\qquad\textrm{whenever }x\in I^n_l\,,\,y\in\R\setminus B_{\rho}(\xi_{\bar\jmath}).
\end{equation}
By \eqref{inftyest} and \eqref{denoest}, we deduce
\begin{equation}\label{incro0}
\begin{aligned}
&\,\bigg| \int_{I^n_l}\ud x \int_{\R\setminus B_{\rho}(\xi_{\bar\jmath})}\frac{u^{n}[k-1](x)+\frac {l}{2}-\big(u^n[k-1](y)+\frac {l} 2\chi_{A^n_l}(y)\big)}{|x-y|^{1+2s}}\ud y\bigg|\\
\le&\, C(s)\Lambda\rho^{-2s}(x^n_{\phi(\xi_{\bar\jmath})+l}-x^{n}_{\phi(\xi_{\bar\jmath})+l-1}).
\end{aligned}
\end{equation}
Now we compute
\begin{equation}\label{incro}
\begin{aligned}
&\int_{I^n_l}\ud x\int_{B_\rho(\xi_{\bar\jmath})\setminus \bar I^n_l}\frac{u^{n}[k-1](x)+\frac {l}{2}-u^n[k-1](y)}{|x-y|^{1+2s}}\ud y\\
=&\,\int_{I^n_l}\ud x\int_{\xi_{\bar\jmath}-\rho}^{x^n_{\phi(\xi_{\bar\jmath})+l-1}}\frac{u^{n}[k-1](x)+\frac {l}{2}-u^n[k-1](y)}{|x-y|^{1+2s}}\ud y\\
&\,+\int_{I^n_l}\ud x\int^{\xi_{\bar\jmath}+\rho}_{x^n_{\phi(\xi_{\bar\jmath})+l}}\frac{u^{n}[k-1](x)+\frac {l}{2}-u^n[k-1](y)}{|x-y|^{1+2s}}\ud y.
\end{aligned}
\end{equation}
Since 
$$
u^{n}[k-1](x)-u^n[k-1](y)=-l+x-y\qquad\textrm{whenever }x\in I^n_l\,,\, y\in(\xi_{\bar\jmath}-\rho,  x^n_{\phi(\xi_{\bar\jmath})+l-1}),
$$
we get
\begin{equation}\label{incro1}
\begin{aligned}
&\int_{I^n_l}\ud x\int_{\xi_{\bar\jmath}-\rho}^{x^n_{\phi(\xi_{\bar\jmath})+l-1}}\frac{u^{n}[k-1](x)+\frac {l}{2}-u^n[k-1](y)}{|x-y|^{1+2s}}\ud y\\
=&\,-\frac l 2\is^s\big(I^n_l,(\xi_{\bar\jmath}-\rho,x^n_{\phi(\xi_{\bar\jmath})+l-1})\big)+
\int_{I^n_l}\ud x\int_{\xi_{\bar\jmath}-\rho}^{x^n_{\phi(\xi_{\bar\jmath})+l-1}}\frac{x-y}{(x-y)^{1+2s}}\ud y\\
\ge &\,-\frac l 2\is^s\big(I^n_l,(\xi_{\bar\jmath}-\rho,x^n_{\phi(\xi_{\bar\jmath})+l-1})\big).
\end{aligned}
\end{equation}
Moreover, we observe that 
$$
u^{n}[k-1](x)-u^n[k-1](y)\ge 1+x-y\qquad\textrm{whenever }x\in I^n_l,\,\,x^{n}_{\phi(\xi_{\bar\jmath})+l}\le y\le\xi_{\bar\jmath}+\rho.
$$
Whence, by arguing as in \eqref{contoausi}, we deduce
\begin{equation}\label{incro2}
\begin{aligned}
&\,\int_{I^n_l}\ud x\int^{\xi_{\bar\jmath}+\rho}_{x^n_{\phi(\xi_{\bar\jmath})+l}}\frac{u^{n}[k-1](x)+\frac {l}{2}-u^n[k-1](y)}{|x-y|^{1+2s}}\ud y\\
\ge&\, \Big(\frac{l}{2}+1\Big)\is^s\big(I^n_l,(x^n_{\phi(\xi_{\bar\jmath})+l}, \xi_{\bar\jmath}+\rho)\big)-C(s)(x^n_{\phi(\xi_{\bar\jmath})+l}-x^n_{\phi(\xi_{\bar\jmath})+l-1}),
\end{aligned}
\end{equation}
where in the last inequality we have used that
\begin{equation}\label{contoausi}
\begin{aligned}
&\,\int_{I^n_l}\ud x\int^{\xi_{\bar\jmath}+\rho}_{x^n_{\phi(\xi_{\bar\jmath})+l}}\frac{x-y}{|x-y|^{1+2s}}\ud y
\le -C(s)(x^n_{\phi(\xi_{\bar\jmath})+l}-x^n_{\phi(\xi_{\bar\jmath})+l-1}).
\end{aligned}
\end{equation}
Finally, by straightforward computations, we obtain
\begin{equation}\label{calcoloint}
\begin{aligned}
\is^s\big(I^n_l,(\xi_{\bar\jmath}-\rho,x^n_{\phi(\xi_{\bar\jmath})+l-1})\big)\le&\,\frac{1}{2s(1-2s)}(x^n_{\phi(\xi_{\bar\jmath})+l}-x^n_{\phi(\xi_{\bar\jmath})+l-1})^{1-2s}\\
&\,+C(s)(x^n_{\phi(\xi_{\bar\jmath})+l}-x^n_{\phi(\xi_{\bar\jmath})+l-1}),\\
\is^s\big(I^n_l,(x^n_{\phi(\xi_{\bar\jmath})+l}, \xi_{\bar\jmath}+\rho)\big)\ge
&\,\frac{1}{2s(1-2s)}(x^n_{\phi(\xi_{\bar\jmath})+l}-x^n_{\phi(\xi_{\bar\jmath})+l-1})^{1-2s}\\
&\,-C(s)(x^n_{\phi(\xi_{\bar\jmath})+l}-x^n_{\phi(\xi_{\bar\jmath})+l-1}).
\end{aligned}
\end{equation}
By \eqref{differene}, \eqref{vienezero2}, \eqref{incro0}, \eqref{incro}, \eqref{incro1}, \eqref{incro2}, \eqref{calcoloint}, we obtain \eqref{claim}.

Let $\hat\jmath\in\{1,\ldots,J\}$, $\hat q\in \{1,\ldots,\mm(\xi_{\hat\jmath})\}$ be such that $x^n_{\phi(\xi_{\hat\jmath})+\hat q}-x^{n}_{\phi(\xi_{\hat\jmath})+\hat q-1}\ge \frac 1 \Lambda\Delta_X(X^n)$; moreover,    let $\hat k$ be given by:
$\hat k:=\sum_{j=1}^{\hat\jmath-1}(\mm(\xi_{\hat\jmath})-1)$ if $\hat q=\mm(\xi_{\hat\jmath})$ and let $\hat k:=\sum_{j=1}^{\hat\jmath-1}(\mm(\xi_{\hat\jmath})-1)+\hat q-1$ if $1\le \hat q\le \mm(\xi_{\hat\jmath})-1$.


By \eqref{claim} and \eqref{secoclaim}, we have that
\begin{multline*}
\liminf_{n\to +\infty}\frac{\E^s(\widehat X^n)-\E^s(X^n)}{\Delta_X(X^n)}
= \liminf_{n\to +\infty}\sum_{k=1}^K\frac{\E^s(X^n[k])-\E^s(X^n[k-1])}{\Delta_X(X^n)}\\
\ge\frac 1 \Lambda \liminf_{n\to +\infty}\frac{\E^s(X^n[\hat k])-\E^s(X^n[\hat k-1])}{x^n_{\phi(\xi_{\hat\jmath})+\hat m}-x^{n}_{\phi(\xi_{\hat\jmath})+\hat m-1}}-C(s,\Lambda)
=+\infty,
\end{multline*}
whence we deduce the claim.
\end{proof}
With Lemma \ref{collision} in hand, we are now able to prove Proposition \ref{nocollis<}.
\begin{proof}[Proof of Proposition \ref{nocollis<}]
Assume by contradiction that the maximal existence time $T$ for the Cauchy problem \eqref{cauchy} is finite. 
Notice that, given $	\bar \delta >0$ there exists $\tau=\tau(\bar \delta)>0$ such that, if $\delta(X(\bar t))\ge \bar \delta$ for some $\bar t\ge 0$, then $T\ge \bar t + \tau$. 
By  definition of maximal time, this implies that
 \begin{equation*}\label{deltaazero}
 \lim_{t \to T^-}\delta(X(t))=0,
 \end{equation*}
which, by Proposition \ref{spostouna}, yields
 \begin{equation}\label{contr}
\lim_{t\to T^-}|\nabla\E^s(X(t))|=+\infty.
\end{equation}

Since $X$ solves the Cauchy problem \eqref{cauchy} and $\E^s$ is bounded from below, we have that $X$ is an absolutely continuous function in $[0,T]$ and that for all $t\in (0,T)$ 
\begin{equation*}
\begin{aligned}
\frac{\E^s(X(T))-\E^s(X(t))}{\dist(X(t),X(T))} & =- \frac{ \int_t^T | \frac{\mathrm{d}}{\mathrm{d} \sigma} \E^s(X(\sigma))| \ud \sigma}{\dist(X(t),X(T))}
\\
& =- \frac{ \int_t^T |\dot X(\sigma)||\nabla\E^s(X(\sigma))| \ud \sigma}{\dist(X(t),X(T))} \le - \inf_{\sigma\in(t,T)} |\nabla\E^s(X(\sigma))|.
\end{aligned}
\end{equation*}
By \eqref{contr} we get
\begin{equation}\label{gradinf}
\lim_{t\to T^-} \frac{\E^s(X(T))-\E^s(X(t))}{\dist(X(t),X(T))}=-\infty.
\end{equation}
Let now $\{t^n\}_{n\in\N}$ be a sequence of non-decreasing times such that $t^n\to  T^-$ (as $n\to +\infty$). For every $n\in\N$ we set $X^n:=X(t^n)$.
Recalling the definition of $\widehat X^n$ in \eqref{hatdefi}, by Lemma \ref{missinglemma}, we have
\begin{equation}\label{prelima}
\lim_{n\to +\infty}\frac{\E^s(X(T))-\E^s(\widehat X^n)}{\dist(\widehat X^n,X(T))}=C\in\R.
\end{equation}
Moreover, by Lemma \ref{collision}
$
\E^s(\widehat X^n)-\E^s(X^n)\ge 0
$
for $n$ large enough.
Therefore, in view of  \eqref{gradinf} and \eqref{crucial}, we have 
\begin{equation}\label{prelima}
\begin{aligned}
-\infty=&\,\lim_{n\to +\infty}\frac{\E^s(X(T))-\E^s(X^n)}{\dist(X^n,X(T))}\\
\ge&\,
\liminf_{n\to +\infty}\frac{\E^s(X(T))-\E^s(\widehat X^n)}{\dist(\widehat X^n,X(T))}\frac{\dist(\widehat X^n,X(T))}{\dist(X^n,X(T))}+\liminf_{n\to +\infty}\frac{\E^s(\widehat X^n)-\E^s(X^n)}{\dist(X^n,X(T))}\\
\ge&\,-|C|\sqrt{\Lambda},
\end{aligned}
\end{equation}
thus getting a contradiction. Therefore $T=+\infty$
and $\delta(X(t))>0$ for every $t\ge 0$.

\end{proof}
For any given $X \in \X^\Lambda$ we set 
\begin{equation}\label{barycenter}
\bari(X):= 
(X\cdot e)/\Z.
\end{equation}
The next result provides the convergence of the gradient flow solution to a global minimizer, with (optimal in $\Lambda$) exponential decay. 
\begin{proposition}\label{convecrit}
Let $X^0\in\rs^\Lambda$ and let $X^\infty$ be the unique (up to permutations of the indices) configuration in $\cs^\Lambda$ having the same barycenter as $X^0$. Then, denoting by $X=X(t)$ the unique solution to the Cauchy problem \eqref{cauchy} we have
\begin{equation*}
\dist(X(t),X^\infty)\le \dist(X^0,X^\infty)e^{-\gamma(s)\Lambda^{-2s}t},
\end{equation*}
where $\gamma(s)$ is the constant in \eqref{stimasubpeggio}.
\end{proposition}
\begin{proof}
We set $\widetilde{\X}^\Lambda:=\X^\Lambda/(\R e)$ and  $\widetilde{\rs}^\Lambda:=\rs^\Lambda/(\R e)$.
For every $X\in{\X}^\Lambda$
 we denote by $\widetilde{\XX}$ the $\Lambda$-periodic extension of $X-(X\cdot e) e$ and by $\widetilde X$ the corresponding element in $\X^{\Lambda}$; in this sense, $\widetilde X$ provides the canonical representative of the equivalence class $[X]$ of $X$ in $\widetilde{\X}^\Lambda$.
 Let $\widetilde \E^s:\widetilde{\X}^\Lambda\to [0,+\infty)$ be defined by
$\widetilde\E^s([X]):=\E^s(\widetilde X)$. 
Since $\nabla \E^s(X+\tau)=\nabla \E^s(X)$ for all $\tau \in\R$ we clearly have  
\begin{equation}\label{gradort}
\nabla \E^s(X)\cdot e=0,
\end{equation}
whence we deduce that $\nabla\widetilde\E^s([X])=\nabla\E^s(X)$.
From now on we will identify $[X]$ with $\widetilde X$.

Let $X=X(t)$ be a solution to \eqref{cauchy}. 
By \eqref{gradort}, we have that
$\dot{(X\cdot e)} = \dot X\cdot e =-\nabla\E^s(X)\cdot e = 0$, so that $\bari(X(t))=\bari(X_0)$ for every $t\ge 0$.
Moreover, 
by Proposition \ref{nocollis<} and by \eqref{gradort} we have that $\dist(X(t),X_\infty)\to 0$ as $t\to +\infty$ .
By the very definition of $\widetilde X$ we have
$$
\dot{\widetilde{X}}= \dot X-(\dot X\cdot e)e=-\big(\nabla \E^s(X)-(\nabla \E^s(X)\cdot e) e\big)=-\nabla\widetilde\E^s(\widetilde X),
$$
so that
$\dist(\widetilde X(t),\widetilde X_\infty)\to 0$ as $t\to +\infty$.
Finally, by \eqref{posvarsecsub} we have that $\widetilde{\E}^s$ is strictly convex on $ \widetilde{\rs}^\Lambda$ and by \eqref{stimasubpeggio}, using Gronwall inequality,
we get that
$$
\dist( X(t), X_\infty)=\dist(\widetilde X(t),\widetilde X_\infty)\le \dist(\widetilde X_0,\widetilde X_\infty)e^{-C(s)\Lambda^{-2s}t}=\dist(X_0, X_\infty)e^{-C(s)\Lambda^{-2s}t}.
$$
\end{proof}
\subsection{The supercritical case $\frac 12 \le s<1$}
With the notation introduced in Section \ref{subsec:dynsub}, we will study the gradient flow system of the energy $\E_\ep^s$ in \eqref{enesupercrit} for  $\frac 1 2\le s<1$ and $\ep>0$
\begin{equation}\label{cauchyep}
\left\{
\begin{array}{l}
\dot{X}(t)=-\nabla\E_\ep^s(X(t))\\
X(0)=X^0,
\end{array}
\right.
\end{equation}
where $X^0\in\X^{\Lambda}$.
Notice that by Proposition \ref{prop:firstvarsuper}, for every $X^0\in\X^\Lambda$, the Cauchy problem \eqref{cauchyep} admits a unique and $C^1$ solution $X_\ep=X_\ep(t)$ in $[0,+\infty)$.

The next result concerns the convergence of the gradient flow solution to a minimizer of the energy.
Before stating it, recall the definitions of $\bari(\cdot)$ and of $\dist(\cdot,\cdot)$ in \eqref{barycenter} and \eqref{distanza}, respectively.
\begin{proposition}
Let $\frac 1 2\le s<1$ and let $X^0\in\rs^{\Lambda}$. There exists $\bar\ep=\bar\ep(s,X^0)>0$ such that, for every $0<\ep<\bar\ep$, denoting by $X_\ep=X_\ep(t)$ the unique solution to \eqref{cauchyep} we have that
 $X_\ep(t)\in \rs^{\Lambda}$ for every $t\ge 0$.
 Moreover, denoting by $X^\infty$ the unique - up to a permutation of indices - element in $\cs^{\Lambda}$ such that $\bari(X^\infty)=\bari (X^0)$, we have that
\begin{equation}\label{rateconv}
\dist(X_\ep(t),X^{\infty})\le \dist(X^0,X^{\infty})e^{-\gamma(s)\Lambda^{-2s}t},
\end{equation}
where $\gamma(s)$ is the constant in \eqref{stimasuperpeggio}.
%
\end{proposition}
\begin{proof}
It is easy to see that \eqref{base} still holds true with $\overline X$ replaced by $X^0$ and $\bar C(s)$ replaced by some constant $\bar C(s,X^0)$. Let $\omega^s(\cdot):=C(s)\sigma^s(\cdot)$ be the quantity on the right-hand side of \eqref{theclaim} and let
 $\delta>0$ be as small as required in Proposition \ref{quasigamma} and such that $\bar C(s,X^0)<\frac{\omega^s(\delta)}{2}$.
 We first show that, for $\ep$ small enough, 
 \begin{equation}\label{deltastacc}
 \delta(X_\ep(t))>\delta\qquad\textrm{ for every }t\ge 0.
 \end{equation}
Indeed, assume   that \eqref{deltastacc} does not hold true at some  time $ t_\e\ge 0$; then,  since  the energy is decreasing along the flow,  Proposition \ref{quasigamma} would yield (for $\ep$ small enough) the following contradiction
 \begin{equation}\label{assu}
 \Lambda\sigma^s(\ep)+\frac{\omega^s(\delta)}{2} \le  \E_\ep^s(X_\ep( t_\e))\le \E_\ep^s(X^0)\le \Lambda\sigma^s(\ep)+\bar C(s,X^0)<  \Lambda\sigma^s(\ep)+\frac{\omega^s(\delta)}{2}\, .
 \end{equation}
 Let $\rs^\Lambda_{\delta}$ denote the class of configurations $X$ satisfying $\delta(X)>\delta$; 
assuming $4\ep<\delta$, we can use \eqref{stimasuperpeggio} to deduce that $\nabla^2\E_\ep^s$ is strictly positive modulo rigid translations on $\rs^\Lambda_{\delta}$, whence, together with \eqref{primasuper} and \eqref{laplafranullo},  we deduce that the set of critical points of $\E_\ep^s$ in $\rs^\Lambda_{\delta}$ coincides with 
$\cs^\Lambda$. 
Finally, by arguing verbatim as in the proof of Proposition \ref{convecrit}, using \eqref{stimasuperpeggio} in place of \eqref{stimasubpeggio}, we obtain 
the second part of the statement.
\end{proof}
Finally, we analyze the behavior, as $\ep\to 0$, of the solutions $X_\ep=X_\ep(t)$ of the Cauchy problem \eqref{cauchyep}.
To this end we introduce the following notation.
 Recalling Remark \ref{bendefi}, we can define for $\frac 1 2\le s<1$, the functional $\F^s:\X^\Lambda\to\R^{\Lambda}$ as
$$
\F^s_\lambda(X):=(-\Delta)^su^X(x_\lambda).
$$
\begin{proposition}
Let $\frac 1 2\le s<1$ and let $X^0=(X^0)^\ord \in\rs^\Lambda$. 
For every $\ep>0$, let $X_\ep=X_\ep(t)$ be the unique solution to the Cauchy problem \eqref{cauchyep}. Then, $X_\ep$ converges uniformly  in $[0,+\infty)$ (as $\ep\to 0$), with respect to the distance $\dist(\cdot,\cdot)$ in \eqref{distanza},
 to the global solution $X_0$ 
of  the Cauchy problem 
\begin{equation}\label{nuovopb}
\left\{\begin{array}{l}\dot X=-\F^s(X)\\
X(0)=X^0.
\end{array}
\right.
\end{equation}
\end{proposition}
\begin{proof}
By \eqref{deltastacc} there exists $\delta>0$ such that, for $\e$ small enough,  the trajectories $X_\e$ lie in the set $\rs^\Lambda_{\delta}$ of configurations $X$ satisfying $\delta(X)>\delta$. By Remark \ref{bendefi} and by \eqref{primasuper}, we get that $\nabla\E_\ep^s$ converge uniformly  to $\F^s$ in $\rs^\Lambda_{\delta}$. As a consequence, the Cauchy problem \eqref{nuovopb} admits a global solution $X_0$ and the trajectories  $X_\e$  converge locally uniformly to $X_0$ as $\e\to 0$. Finally, in view of the uniform rate of convergence \eqref{rateconv} of $X_\e(t)$ to $X^\infty$ as $t\to +\infty$, the convergence of   $X_\e$ to  $X_0$ is uniform on the whole half-line $[0,+\infty)$.
\end{proof}
\begin{remark}
\rm{Let $\frac 1 2\le s<1$. 
Given $\delta>0$ and recalling that  $\rs^\Lambda_\delta$ denotes the class of configurations $X$ satisfying $\delta(X)>\delta$, we have that $\nabla\E_\ep^s$ converges, uniformly in $\rs^\Lambda_\delta$, to $\F^s$ (as $\ep\to 0$).
Therefore, $\E_\ep^s-C(\ep,s,\Lambda)$, for suitable choices of the constant $C(\ep,s,\Lambda)$, are uniformly bounded (and uniformly Lipschitz) in $\rs^\Lambda_\delta$. Clearly, the constant $C(\ep,s,\Lambda)$ can be chosen equal to $\E^s_\ep(\cs^\Lambda)$; in fact, one can check that
$\E^s_\ep(\cs^\Lambda)-\Lambda\sigma^s(\ep)$ converge (as $\ep\to 0$) to a constant, so that $C(\ep,s,\Lambda)$ could equivalently be chosen equal to $\Lambda\sigma^s(\ep)$. Therefore, setting $\W_\ep^s:=\E_\ep^s-\Lambda\sigma^s(\ep)$, by Ascoli-Arzel\'a Theorem, we have that, up to a subsequence, $\W_\ep^s$ converge uniformly in $\rs^\Lambda_\delta$ to a function $\W_0^s$ (as $\ep\to 0$). Moreover, by uniform convergence of the gradients, $\nabla\W_0^s=\F^s$.
As a consequence, the limit function $\W^s_0$ is uniquely determined, up to additive constants, and in fact, using that the whole sequence $\W_\ep^s(\cs^\Lambda)$ converges as $\ep\to 0$, we have that $\W^s_0$ is uniquely determined. As a consequence, the whole sequence $\W_\ep^s$ converges to $\W_0^s$ as $\ep\to 0$. In this respect, the solution $X_0$ of \eqref{nuovopb} is the gradient flow of the {\it renormalized energy} $\W_0^s$ starting from $X^0$.
}
\end{remark}


\begin{thebibliography}{99}

\bibitem{BL} Blanc, X., Lewin, M.: The crystallization conjecture: a review. {\it EMS Surv. Math. Sci.} {\bf 2} (2015), 255--306.

\bibitem{BS}{Borodin, A., Serfaty, S.}: Renormalized energy concentration in random matrices. {\it Commun. Math. Phys.} {\bf 320} (2013), 199--244.
%


%

\bibitem{DPS}{De Luca, L., Ponsiglione, M., Spadaro, E.N.}: Two slope functions minimizing fractional seminorms and applications to misfit dislocations. {\it SIAM J. Math. Anal.} {\bf 56} (2024), 1179--1196.

\bibitem{FPS} {Fanzon, S., Ponsiglione, M., Scala, R.}: Uniform distribution of dislocations in Peierls-Nabarro models for semi-coherent interfaces. {\it Calc. Var. Partial Differ. Equ.} {\bf 59} (2020), art. n. 141




\bibitem{LG}
{Giuliani, A., Lebowitz, J.L., Lieb, E.H.}: Periodic minimizers in 1D local mean field theory. {\it Commun. Math. Phys. } {\bf 286} (2009), 163--177

\bibitem{GM}
{Giuliani, A., M\"uller, S.}: Striped periodic minimizers of a two-dimensional model for martensitic phase transitions. {\it Commun. Math. Phys.} {\bf 309} (2012), 313--339

\bibitem{L} Lebl\'e, T.: A Uniqueness Result for Minimizers of the 1D Log-gas
Renormalized Energy. {\it  Funct. Anal.} {\bf 268} (2015), 1649--1677.
\bibitem{Lewin} Lewin, M.: Coulomb and Riesz gases: The known and the unknown. {\it Journal of Mathematical Physics} 63.6 (2022).


\bibitem{M93} M\"uller, S.: Singular perturbations as a selection criterion for periodic minimizing sequences. {\it Calc. Var. Partial Differ. Equ.} {\bf 1} (1993), 169--204

%

\bibitem{RS} {Read, W.T., Shockley, W.}: Dislocation models of crystal grain boundaries. {\it Phys. Rev.} {\bf 78} (1950), 275--289

\bibitem{RW}{Ren, X., Wei, J.}: On energy minimizers of the diblock copolymer problem. {\it Interf. Free Bound.} {\bf 5} (2003), 193--238


\bibitem{PS} Petrache, M., Serfaty, S.: Crystallization for coulomb and riesz interactions as a consequence
of the Cohn-Kumar conjecture. {\it Proc. American Math. Soc.} {\bf 148} (2020), 3047--3057.

\bibitem{SS} Sandier, E., Serfaty, S.: 1d log gases and the renormalized energy: crystallization at vanishing
temperature. {\it Probab. Theory and Related Fields} {\bf 162} (2015), 795--846.
\bibitem{Serfatybook} Serfaty, S.: Lectures on Coulomb and Riesz gases. arXiv preprint arXiv:2407.21194 (2024).
\bibitem{VdM} {van der Merwe, J.}: On the stresses and energies associated with inter-crystalline boundaries. {\it Proc.
Phys. Soc. A} {\bf 63} (1950), 616--637


\end{thebibliography}
\end{document}